\documentclass{compositio}

\usepackage{amsthm}
\usepackage{amssymb}
\usepackage{amsmath}
\usepackage{amsfonts} 
\usepackage{amsthm}
\usepackage{bm}
\usepackage{color}
\usepackage{graphicx}
\usepackage[all]{xy}

\newcommand{\excise}[1]{}

\newtheorem{theorem}{Theorem}[section]
\newtheorem{lemma}[theorem]{Lemma}

\newtheorem{corollary}[theorem]{Corollary}
\newtheorem{proposition}[theorem]{Proposition}

\theoremstyle{definition}
\newtheorem{example}[theorem]{Example}
\newtheorem{remark}[theorem]{Remark}

\newtheorem{definition}[theorem]{Definition}

\newtheorem{observation}[theorem]{Observation}
\newtheorem{notation}[theorem]{Notation}

\newenvironment{romanlist}
        {\renewcommand\theenumi{\roman{enumi}}\begin{list}
                {\noindent\makebox[0mm][r]{(\roman{enumi})}}
                {\leftmargin=5.5ex \usecounter{enumi}}
        }
        {\end{list}\renewcommand\theenumi{\arabic{enumi}}}

\newcommand{\baseRing}[1]{\ensuremath{\mathbb{#1}}}
\newcommand{\Z}{\baseRing{Z}}
\newcommand{\C}{\baseRing{C}}
\newcommand{\N}{\baseRing{N}}
\newcommand{\R}{\baseRing{R}}
\newcommand{\Q}{\baseRing{Q}}

\def\<{\langle}
\def\>{\rangle}
\def\0{\mathbf{0}}

\def\CC{{\mathbb C}}
\def\EE{{\mathcal E}}
\def\FF{{\mathcal F}}
\def\HH{{\mathcal H}}
\def\KK{{\mathcal K}}

\def\PP{{\mathbb P}}

\def\RR{{\mathbb R}}
\def\SS{{\mathbb S}}
\def\TT{{\mathbb T}}
\def\ZZ{{\mathbb Z}}

\def\cC{{\mathcal C}}
\def\cG{{\mathrm \Delta}}
\def\cI{{\mathcal I}}
\def\cJ{{\mathcal J}}

\def\cP{{\mathcal P}}

\def\cR{{\mathcal R}}
\def\cS{{\mathcal S}}

\def\cZ{{\mathcal Z}}
\def\bbE{{\mathbb E}}

\def\dM{{\mathbb M}}
\def\fA{{\mathfrak A}}
\def\fC{{\mathfrak C}}
\def\fS{{\mathfrak S}}
\def\oalpha{\overline{\alpha}}
\def\del{\partial}
\def\qdeg{{\rm qdeg}}

\def\ext{\operatorname{Ext}}
\def\mor{\operatorname{Mor}}

\def\vol{{\rm vol}}

\def\rank{{\rm rank\ }}
\def\supp{\text{\rm supp}}

\def\image{\text{\rm image }}
\def\codim{\text{\rm codim}}

\def\nothing{\varnothing}

\def\ol#1{{\overline {#1}}}
\def\wt#1{{\widetilde {#1}}}
\def\dlim{{\lim\limits_{\raisebox{.2ex}
    {$\scriptstyle\longrightarrow$}}}}

\numberwithin{equation}{section}
\begin{document}

\title{The Rank of a Hypergeometric System}
\author{Christine Berkesch}
\email{cberkesc@purdue.edu}
\address{Department of Mathematics\\ Purdue University\\
West Lafayette, IN \ 47907.}
\thanks{This author was partially supported by NSF Grant DMS 0555319.}
\subjclass[2000]{
33C70, 
14M25, 16E30, 20M25, 13N10}
\keywords{hypergeometric, D-module, toric, 
	holonomic, rank, Euler--Koszul, combinatorial.}

\begin{abstract}
The holonomic rank of the $A$-hypergeometric system $M_A(\beta)$ 
is the degree of the toric ideal $I_A$ for generic parameters; 
in general, this is only a lower bound.
To the semigroup ring of $A$ we attach the ranking arrangement 
and use this algebraic invariant and the exceptional arrangement 
of nongeneric parameters to construct a combinatorial formula for the rank jump of 
$M_A(\beta)$. 
As consequences, we obtain a refinement of the stratification 
of the exceptional arrangement by the rank of $M_A(\beta)$ and 
show that the Zariski closure of each of its strata 
is a union of translates of linear subspaces of the parameter space. 
These results hold for generalized $A$-hypergeometric systems 
as well, where the semigroup ring of $A$ 
is replaced by a nontrivial weakly toric module $M\subseteq \C[\Z A]$. 
We also provide a direct proof of the main result in \cite{saito isom} regarding 
the isomorphism classes of $M_A(\beta)$.
\end{abstract}

\maketitle
\setcounter{tocdepth}{1}
\tableofcontents

\section{Introduction}
\label{sec:intro}

An $A$-hypergeometric system $M_A(\beta)$ is a $D$-module 
determined by an integral matrix $A$  
and a complex parameter vector $\beta \in\C^d$. 
These systems are also known as {\it GKZ-systems}, 
as they were introduced in the late 1980's by Gelfand, 
Graev, Kapranov, and Zelevinsky \cite{GGZ,GKZ}. 
Their solutions occur naturally in mathematics and physics,
including the study of roots of polynomials, 
Picard--Fuchs equations for the variation of
Hodge structure of Calabi-Yau toric hypersurfaces,
and generating functions for intersection numbers on moduli spaces of curves, 
see \cite{sturmfels-solving, BvS95, HLY96, Oko02}.

The (holonomic) rank of $M_A(\beta)$ coincides with 
the dimension of its solution space at a nonsingular point. 
In this article, we provide a combinatorial formula for the rank of $M_A(\beta)$ in 
terms of certain lattice translates determined by $A$ and $\beta$. 
For a fixed matrix $A$, 
this computation yields a geometric stratification of the parameter space $\C^d$ 
that refines its stratification by the rank of $M_A(\beta)$.  

\begin{notation}
Let $A=[ a_1 \ a_2 \ \cdots \ a_n]$ be an integer 
$(d\times n)$-matrix with integral column span $\Z A=\Z^d$.
Assume further that $A$ is {\it pointed}, 
meaning that the origin is the only linear subspace of the cone 
$\R_{\geq0}A= \{\sum_{i=1}^n \gamma_ia_i \mid
    \gamma_i \in\R_{\geq0}\}$.  

A subset $F$ of the column set of $A$ is called a {\it face} of $A$, 
denoted $F\preceq A$, if 
$\R_{\geq0}F$ 
is a face of the cone $\R_{\geq0}A$.

Let $x=x_1,\dots,x_n$ be variables and $\del=\del_1,\dots,\del_n$
their associated partial differentiation operators.
In the polynomial ring $R=\C[\del]$, let 
\begin{align*}
I_A = \<\del^u-\del^v \mid Au=Av,\ u,v\in\N^n \>\subseteq R 
\end{align*}
denote the toric ideal associated to $A$, and let 
$S_A = R/I_A$ be its quotient ring. 
Note that $S_A$ is isomorphic to the semigroup ring of $A$, which is 
\begin{align}\label{eq-semigroup ring}
S_A \cong 
\C[\N A] := \bigoplus_{a\in \N A} \C\cdot t^a 
\end{align} 
with multiplication given by semigroup addition of exponents. 
The Weyl algebra
\[
D = \C\<x,\del \mid [\del_i,x_j]=\delta_{ij}, [x_i,x_j]=0=[\del_i,\del_j]\>
\]
is the ring of $\C$-linear differential operators on $\C[x]$.
\end{notation}

\begin{definition}\label{def-gkz}
The {\it $A$-hypergeometric system} with parameter $\beta\in\C^d$ 
is the left $D$-module
\[
M_A(\beta) = D/D\cdot(I_A,\{E_i-\beta_i\}_{i=1}^d),
\]
where $E_i-\beta_i =\sum_{j=1}^na_{ij}x_j\del_j - \beta_i$
are \emph{Euler operators} associated to $A$.
\end{definition}

The {\it rank} of a left $D$-module $M$ is
\[
\rank M=\dim_{\C(x)}\C(x)\otimes_{\C[x]}M.
\]
The rank of a holonomic $D$-module is finite and equal to the 
dimension of its solution space of germs of holomorphic functions
at a generic nonsingular point \cite{Kashiwara}.

\subsection{The exceptional arrangement of a hypergeometric system}
\label{subsec:background}

In \cite{GKZ}, Gelfand, Kapranov, and Zelevinsky showed that 
when $S_A$ is Cohen--Macaulay and standard $\Z$-graded, 
the $A$-hypergeometric system $M_A(\beta)$ is 
holonomic of rank $\vol(A)$ for all parameters $\beta$, 
where $\vol(A)$ is $d!$ times the Euclidean volume of 
the convex hull of $A$ and the origin.
Adolphson established further that $M_A(\beta)$ is holonomic 
for all choices of $A$ and $\beta$ and that the holonomic rank of 
$M_A(\beta)$ is generically given by $\vol(A)$ \cite{Ado}.
However, an example found by Sturmfels and Takayama showed that 
equality need not hold in general \cite{0134} (see also \cite{SST}). 
At the same time, Cattani, D'Andrea, and Dickenstein produced 
an infinite family of such examples through a complete investigation 
of the rank of $M_A(\beta)$ when $I_A$ defines a 
projective monomial curve \cite{CDD}.

The relationship between $\vol(A)$ and the rank of $M_A(\beta)$ 
was made precise by Matusevich, Miller, and Walther, who used 
Euler--Koszul homology to study the holonomic 
\emph{generalized $A$-hypergeometric system} $\HH_0(M,\beta)$ 
associated to a toric module $M$ (see Definition~\ref{def:toric}).
The Euler--Koszul homology $\HH_\bullet(M,\beta)$ 
of $M$ with respect to $\beta$ 
is the homology of a twisted Koszul complex 
on $D\otimes_R M$ given by the sequence $E - \beta$. 
This includes the $A$-hypergeometric system
$M_A(\beta) = \HH_0(S_A,\beta)$ as the special case that 
$M = S_A$. 
As in this special case, and for the purposes of this article, 
suppose that the generic rank of $\HH_0(M,-)$ is $\vol(A)$. 

The matrix $A$ induces a natural $\Z^d$-grading on $R$;
the {\it quasidegree set} of a finitely generated $\Z^d$-graded 
$R$-module $N$ is defined to be the Zariski closure in $\C^d$ of the set 
of vectors $\alpha$ for which the graded piece $N_\alpha$ is nonzero. 
In \cite{MMW}, an explicit description of the {\it exceptional arrangement}
\[
\EE_A(M) = \{\beta\in\C^d \mid \rank \HH_0(M,\beta)\neq\vol(A)\}
\]
associated to $M$ is given in terms of the quasidegrees of certain 
$\ext$ modules involving $M$ (see \eqref{13}). 
This description shows that $\EE_A(M)$ is a subspace arrangement in 
$\C^d$ given by translates of linear subspaces that are generated by 
the faces of the cone $\R_{\geq0}A$, and that $\EE_A(M)$ is empty exactly 
when $M\neq 0$ is a maximal Cohen--Macaulay $S_A$-module.
It is also shown in \cite{MMW} that the rank of $\HH_\bullet(M,\beta)$ is 
upper semi-continuous as a function of $\beta$. Thus 
the exceptional arrangement $\EE_A(M)$ is the nested union over 
$i\geq 0$ of the Zariski closed sets
\[
\EE_A^i(M) = \{\beta\in\C^d\mid \rank \HH_0(M,\beta)-\vol(A)> i\}. 
\]
In particular, the rank of $\HH_0(M,\beta)$ induces a stratification of $\EE_A(M)$, 
which we call its {\it rank stratification}.

\subsection{A homological study of rank jumps}
\label{subsec:hom study}

The present article is a study of the rank stratification of $\EE_A(M)$ 
when $M\subseteq S_A[\del_A^{-1}]$ is $\Z^d$-graded such that the 
degree set $\dM = \deg(M)$ of $M$ is a nontrivial $\N A$--monoid.  
In particular, $M$ is weakly toric (see Definition~\ref{def-wtoric}). 
If $\dM = \N A$, then $M$ is the semigroup ring $S_A$ from \eqref{eq-semigroup ring} 
and $\HH_0(M,\beta) = M_A(\beta)$ is the $A$-hypergeometric system at $\beta$. 
The module $M$ could also be a localization of $S_A$ along a  
subset of faces of $A$. 
As $M$ will be fixed throughout this article, 
we will often not indicate dependence on $M$ in the notation. 

Examination of the long exact sequence in Euler--Koszul homology
induced by the short exact sequence of weakly toric modules 
\begin{eqnarray*}
    0 \rightarrow M \rightarrow
    S_A[\del_A^{-1}] \rightarrow Q \rightarrow 0
\end{eqnarray*}
reveals that the {\it rank jump} of $M$ at $\beta$,
\begin{align}
j(\beta) & =\rank\HH_0(M,\beta)-\vol(A), \nonumber 
\intertext{can be calculated in terms of $Q$ by}
\label{eqn-Q}
j(\beta) & = \rank \HH_1(Q,\beta)-\rank \HH_0(Q,\beta).
\end{align}
We define the \emph{ranking arrangement} $\cR_A(M)$ of $M$ to be 
the quasidegrees of $Q$. 
Vanishing properties of Euler--Koszul homology imply that the 
exceptional arrangement $\EE_A(M)$ is contained in the 
ranking arrangement $\cR_A(M)$. 
We show in Theorem~\ref{thm-str comparison}
that $\cR_A(M)$ is the union of $\EE_A(M)$ 
and an explicit collection of hyperplanes. 

For a fixed $\beta\in\EE_A(M)$, we then proceed to compute $j(\beta)$. 
In Section~\ref{sec:r-t mods}, 
we combinatorially construct a \emph{finitely generated} $\Z^d$-graded 
\emph{ranking toric module} $P^{\beta}$ with 
$\HH_\bullet(Q,\beta) \cong \HH_\bullet(P^{\beta},\beta)$. 
Since $j(\beta)$ is determined by the Euler--Koszul homology of $Q$ 
by \eqref{eqn-Q}, we see that $P^{\beta}$ contains the information 
essential to computing the rank jump $j(\beta)$. 
To outline the construction of the module $P^{\beta}$, let 
\[
    \FF(\beta)
        = \{ F \preceq A \mid \beta + \C F \subseteq \cR_A(M) \}
\]
be the polyhedral complex of faces of $A$ determined by the components of 
the ranking arrangement $\cR_A(M)$ that contain $\beta$. 
We call the collection of integral points 
\[
\bbE^{\beta} = 
\Z^d \cap \bigcup_{F\in\FF(\beta)} (\beta+\C F)\setminus(\dM+\Z F) 
\]
the \emph{ranking lattices} $\bbE^{\beta}$ of $M$ at $\beta$. 
This set is a union of translates of lattices generated by faces of $A$, 
where the vectors in these lattice translates of $\Z F$ in $\bbE^{\beta}$ 
are precisely the degrees of $Q$ 
which cause $\beta + \C F$ to lie in the ranking arrangement. 
Since it contains full lattice translates, 
$\bbE^{\beta}$ cannot be the degree set of a finitely generated $S_A$-module. 
Thus, to complete the construction of the degree set $\PP^\beta$ of $P^{\beta}$, 
we intersect $\bbE^{\beta}$ with an appropriate half space (see Definition~\ref{def-CP}). 
To give a flavor of our approach for $\beta\in\R^d$, 
this is equivalent to intersecting $\bbE^\beta$ with 
$\cC_A(\beta) = \ZZ^d \cap [\beta + \R_{\geq0}A]$. 
By setting up the proper module structure, 
$\PP^{\beta} = \cC_A(\beta) \cap \bbE^\beta$ gives the 
$\Z^d$-graded degree set of the desired toric module $P^{\beta}$ with 
$j(\beta) = \rank \HH_1(P^{\beta},\beta)-\rank \HH_0(P^{\beta},\beta)$.
After translating the computation of the rank jump $j(\beta)$ to $P^{\beta}$, 
we obtain a generalization of the formula given by Okuyama 
in the case $d=3$ \cite{okuyama}. 

\begin{theorem}\label{thm-compute jump}
The rank jump $j(\beta)$ of $M$ at $\beta$ can be computed from 
the combinatorics of the ranking lattices $\bbE^{\beta}$ of $M$ at $\beta$.
\end{theorem}

In particular, the rank of the hypergeometric system is the same 
at parameters which share the same ranking lattices. 
The proof of Theorem~\ref{thm-compute jump} can be found in 
Section~\ref{subsec:rank jumps} as a special case of our main result, 
Theorem~\ref{thm-main}. 

\subsection{The ranking slab stratification of the exceptional arrangement}
\label{subsec:r-slabs}

Let $X$ and $Y$ be subspace arrangements in $\C^d$. 
We say that a stratification $\cS$ of $X$ \emph{respects} $Y$ 
if for each irreducible component $Z$ of $Y$ and each stratum $S\in \cS$, 
either $S\cap Z = \nothing$ or $S\subseteq Z$. 
A \emph{ranking slab} of $M$ is 
a stratum in the coarsest stratification of $\EE_A(M)$ 
that respects the arrangements $\cR_A(M)$ 
and the negatives of the quasidegrees of each of the 
$\ext$ modules that determine $\EE_A(M)$ (see Definition~\ref{def-ranking slab}).

Proposition~\ref{prop-zc E} states that the parameters 
$\beta,\beta'\in\C^d$ belong to the same 
ranking slab of $M$ exactly when their ranking lattices coincide, 
that is, $\bbE^{\beta} = \bbE^{\beta'}$. 
Combining this with Theorem~\ref{thm-compute jump}, we see that 
the rank of $\HH_0(M,-)$ is constant on each ranking slab. 

\begin{corollary}\label{cor-strat vs equiv}
The function $j(-)$ is constant on each ranking slab. 
In particular, the stratification of the exceptional arrangement 
$\EE_A(M)$ by ranking slabs refines its rank stratification.
\end{corollary}

Hence, like $\EE_A(M)$, each set $\EE_A^i(M)$ is a union of translated 
linear subspaces of $\C^d$ which are generated by faces of $\R_{\geq0}A$.
In order for the ranking slab stratification of $\EE_A(M)$ to refine 
its rank stratification, it must respect 
each of the arrangements appearing in its definition; 
this can be seen from 
Examples~\ref{example-hidden},~\ref{example-nonconstant slab},~and~\ref{plane-line-4diml}. 
In particular, as $\cR_A(M)\supsetneq\EE_A(M)$, 
the exceptional arrangement $\EE_A(M)$ 
does not generally contain enough information to 
determine its rank stratification. 

\subsection{A connection to the isomorphism classes of hypergeometric systems}%
\label{subsec:iso conn}

When $M = S_A$, the ranking lattices $\bbE^{\beta}$
are directly related to the combinatorial sets $E_\tau(\beta)$ defined by Saito, 
which determine the isomorphism classes of $M_A(\beta)$. 
In \cite{saito isom,ST diffl algs}, various $b$-functions 
arising from an analysis of the symmetry algebra of 
$A$-hypergeometric systems are used to link 
these isomorphism classes to the sets $E_\tau(\beta)$. 
We conclude this paper with a shorter proof, 
replacing the use of $b$-functions with Euler--Koszul homology.

\subsection*{Outline}
\label{subsec:outline}

The following is a brief outline of this article. 
In Section~\ref{sec:defns}, we summarize definitions and results on
weakly toric modules and Euler--Koszul homology, following \cite{MMW,EKDI}.
Section~\ref{sec:EKstr} is a study of the structure of the Euler--Koszul complex 
of maximal Cohen--Macaulay toric face modules. 
The relationship between the exceptional and ranking arrangements of $M$ 
is made precise in Section~\ref{sec:exceptional arrangement}. 
In Section~\ref{sec:r-t mods}, we define the class of ranking toric modules, 
which play a pivotal role in calculating the rank jump $j(\beta)$. 
Section~\ref{sec:comb of rank} contains our main theorem, 
Theorem~\ref{thm-main}, which results in the computation 
of $j(\beta)$ for a fixed parameter $\beta$.
We close with a discussion on the isomorphism classes of 
$A$-hypergeometric systems in Section~\ref{sec:isom classes}.

\subsection*{Acknowledgments}

I am grateful to my advisor Uli Walther for many inspiring conversations 
and thoughtful suggestions throughout the duration of this work. 
I would also like to thank Laura Felicia Matusevich for asking the question 
``Is rank constant on a slab?" as well as helpful remarks on this text  
and conversations that led directly to 
the proof of Theorem~\ref{thm-isom classes}. 
Finally, my thanks to Ezra Miller for many valuable comments 
for improving this paper, 
including the addition of Definition~\ref{def-cellular definition}, 
and to the referee for detailed notes that clarified this article.

\section{The language of Euler--Koszul homology}
\label{sec:defns}

In this section, we summarize definitions found in the literature and set notation.
Most important are the definitions of a weakly toric 
module \cite{EKDI} and Euler--Koszul homology \cite{MMW}.

Let $a_1,a_2,\dots,a_n$ denote the columns of $A$. 
For a face $F\preceq A$, let $F^c$ denote the complement of a face $F$ 
in the column set of $A$. 
If $F$ is any subset of the columns of $A$, 
the \emph{codimension} of $F$ is $\codim(F):= \codim_{\C^d}(\C F)$, 
the codimension of the $\CC$-vector space generated by $F$. 
The \emph{dimension} of $F$ is $\dim(F)=d-\codim_\C(\C F)$. 

A face $F$ of $A$ is a {\it facet} of $A$ if $\dim(F)=d-1$. 
Recall that the {\it primitive integral support function} of a facet 
$F\preceq A$ is the unique linear functional 
$p_F:\C^d\rightarrow\C$ such that
    \begin{romanlist}
      \item $p_F(\Z A)= \Z,$
      \item $p_F(a_i)\geq0$ for all $j=1,\dots,n$, and
      \item $p_F(a_i)=0$ exactly when $a_i\in F$.
    \end{romanlist}
The {\it volume} of a face $F$, denoted $\vol(F)$, is the integer 
$\dim(F)!$ times the Euclidean volume in $\Z F \otimes_{\Z} \R$ 
of the convex hull of $F$ and the origin.

\begin{definition}\label{def:semigroup ring D}
Let $\N F = \{\sum_{a_i\in F} \gamma_ia_i \mid \gamma_i \in\N \}$
be the semigroup generated by the face $F$ and, as in \eqref{eq-semigroup ring},
\[
S_F=\C[\N F]
\] 
is the corresponding semigroup ring, called a {\it face ring} of $A$. 
Let $x_F = \{ x_i \mid a_i\in F\}$ and $\del_F = \{\del_i\mid a_i\in F\}$. 
Define 
\[
R_F = \C[\del_F] 
\]
to be the polynomial ring in $\del_F$ and 
\[
D_F = \C\<x_F,\del_F\mid [x_i,\del_j] = \delta_{ij},[x_i,x_j] = 0 = [\del_i,\del_j]\>
\]
to be the \emph{Weyl algebra} associated to $F$. 
Note that 
\[
S_F \cong R_F/ (I_F +\<\del_{F^c}\>) \quad \text{ with } I_F=\ker(R_F\rightarrow S_F) \text{ and } F^c = A\setminus F.
\]
\end{definition}

\begin{definition}\label{def:monoid}
Let $t=t_1,\dots,t_d$ be variables. 
For a face $F\preceq A$, we say that a subset $\SS\subseteq\Z^d$ is an 
\emph{$\N F$--module} if $\SS + \N F\subseteq \SS$. 
Further, we call an $\N F$--module $\SS$ an \emph{$\N F$--monoid} 
if it is closed under addition, that is, for all $s,s'\in\SS$, $s+s'\in\SS$. 
Given an $\N F$--module $\SS$, define the $S_F$-module
$\C\{\SS\} = \bigoplus_{s\in\SS}\C\cdot t^s$ as a $\C$-vector space 
with $S_F$-action given by $\del_i\cdot t^s = t^{s + a_i}$. 
Further, $\C\{\SS\}$ is equipped with a multiplicative structure given by 
$t^s\cdot t^{s'} = t^{s+s'}$ for $s,s'\in\SS$ and extended $\C$-linearly. 
By definition, $\N F$ is an $\N F$--monoid and $S_F\cong \C\{ \N F\}$ as rings.
\end{definition}

Define a $\Z^d$-grading on $R_F \subseteq D_F$ by setting 
\[
\deg(\del_i)=a_i \quad \text{and} \quad \deg(x_i)=-a_i.
\]
Then $\C\{\SS\}$ is naturally a $\Z^d$-graded $S_F$-module 
by setting $\deg(t^s)=s$. 

The {\it saturation} of $F$ in $\Z F$ is the semigroup
$\wt{\N F}=\R_{\geq 0}F\cap\Z F$.
The {\it saturation}, or {\it normalization}, of $S_F$ is the semigroup ring 
of the saturation of $F$ in $\Z F$, which is given by $\wt{S}_F=\C\{ \wt{\N F} \}$ 
as a $\Z^d$-graded $S_F$-module. 
Note that $\wt{S}_F$ is a Cohen--Macaulay $S_F$-module \cite{Hochster}.

If $N$ is a $\Z^d$-graded $R$-module and $v\in\Z^d$,
the {\it degree set} of $N$, denoted $\deg(N)$,
is the set of all $v\in\Z^d$ such that $N_v\neq0$.
Let $N(v)$ denote the $\Z^d$-graded module with $v'$-graded piece 
$N(v)_{v'} = N_{v+v'}$. 

We now recall the definitions of toric and weakly toric modules and 
their quasidegree sets, which can be found in 
\cite[Definition~4.5]{MMW} and \cite[Section~5]{EKDI}, respectively. 

\begin{definition}\label{def:toric}
A $\Z^d$-graded $R$-module is \emph{toric} if it has a filtration 
\[
0 = M_0 \subseteq M_1 \subseteq \cdots \subseteq 
M_{\ell - 1} \subseteq M_\ell = M
\]
such that for each $i$, $M_i/M_{i-1}$ is a $\Z^d$-graded translate of 
$S_{F_i}$ for some face $F_i \preceq A$. 
Notice that toric modules are necessarily finitely generated $R$-modules. 
\end{definition}

\begin{definition}\label{def-qdeg}
If $N$ is a finitely generated $\Z^d$-graded $R$-module, 
a vector $v\in\C^d$ is a {\it quasidegree} of $N$, written
$v\in\qdeg(N)$, if $v$ lies in the Zariski closure of
$\deg(N)$ under the natural embedding $\Z^d\hookrightarrow\C^d$. 
Notice that if $N$ is toric, then $\qdeg(N)$ is a finite subspace arrangement 
in $\C^d$, consisting of translated subspaces generated by faces of $A$, 
see \cite[Lemma~2.5]{DMM}. 
\end{definition}

A partially ordered set $(\fS,\leq)$ is \emph{filtered} 
if for each $s',s''\in\fS$ there exists $s\in\fS$ 
with $s'\leq s$ and $s''\leq s$. 

\begin{definition}\label{def-wtoric}
We say that a $\Z^d$-graded $R$-module $M$ is \emph{weakly toric} if 
there is a filtered partially ordered set $(\fS,\leq)$
and a $\Z^d$-graded direct limit 
\[
\phi_s: M_s \to {\dlim}_{s\in\fS} M_s = M
\]
where $M_s$ is a toric $R$-module for each $s\in\fS$. 
We then define the \emph{quasidegree set} of $M$ to be 
\[
\qdeg(M) = \bigcup_{s\in\fS} \qdeg(\phi_s(M_s)),
\]
where each $\qdeg(\phi_s(M_s))$ is defined by Definition~\ref{def-qdeg}. 
\end{definition}

\begin{example}\label{ex-M wt}
If $\dM\subseteq\Z^d$ is an $\N A$--module, 
then $M = \C\{\dM\}$ is weakly toric 
because it is a direct limit over $b\in\dM$ of $S_A(-b)$ under 
the natural $A$-homogeneous inclusion 
$S_A(-b)\subseteq S_A[\del_A^{-1}]\cong \C[\Z^d]$. 
\end{example}

\begin{example}
Consider the matrix 
$A = 
\left[\begin{smallmatrix}
1 & 1 & 1 & 1 \\
0 & 1 & 3 & 4
\end{smallmatrix}\right]$
with face $F = \left[\begin{smallmatrix}
1 \\
4
\end{smallmatrix}\right]$. 
The module $S_F[\del_F^{-1}]$ is weakly toric with quasidegree set 
\[
\qdeg\left(S_F[\del_F^{-1}]\right) = \C F 
\]
because it is a filtered direct limit over $b\in\Z F$ of $S_F(-b)$. 
Similarly, the module $S_A[\del_F^{-1}]$ is weakly toric 
with $\qdeg\left(S_A[\del_F^{-1}]\right) = \C^2$. 
The quotient $S_A[\del_F^{-1}]/S_A$ is also weakly toric. 
Its quasidegree set consists of the point 
$\left[\begin{smallmatrix} 1 \\ 2 \end{smallmatrix}\right]$ 
and the union of lines in $\{ t_2 = k \mid k\in\Z_{<0}\}$, 
where $(t_1,t_2)$ are the coordinates of $\C^2$. 
\end{example}

We now recall the definition of the Euler--Koszul complex of a weakly toric module 
$M$ with respect to a parameter $\beta\in\C^d$. 
For $1\leq i\leq d$, each {\it Euler operator} 
$E_i-\beta_i = \sum_{j=1}^na_{ij}x_j\del_j - \beta_i$
determines a $\Z^d$-graded $D$-linear endomorphism of $D\otimes_R M$, 
defined on a homogeneous $y\in D\otimes_R M$ by
\begin{eqnarray*}
(E_i-\beta_i)\circ y=(E_i - \beta_i + \deg_i(y))y
\end{eqnarray*}
and extended $\C$-linearly.  
This sequence $E-\beta$ of commuting endomorphisms determines 
a Koszul complex $\KK^A_\bullet(M,\beta)=\KK_\bullet(M,\beta)$
on the left $D$-module $D\otimes_R M$, called the 
{\it Euler--Koszul complex} of $M$ with parameter $\beta$. 
The {\it $i^\text{th}$ Euler--Koszul homology module} of $M$ is
$\HH^A_i(M,\beta)=\HH_i(M,\beta) = H_i(\KK_\bullet(M,\beta))$. 
Our object of study will be the {\it generalized $A$-hypergeometric system}
$\HH_0(M,\beta)$ associated to $M$.

The Euler--Koszul complex defines an exact functor from the category of 
weakly toric modules with degree-preserving morphisms to the category of 
bounded complexes of $\Z^d$-graded left $D$-modules with degree-preserving 
morphisms, so short exact sequences of weakly toric modules yield long exact 
sequences of Euler--Koszul homology.
Notice also that Euler--Koszul homology behaves well under $\Z^d$-graded 
translations: for $b\in\Z^d$,
\begin{eqnarray}\label{shift HH}
\HH_q(M(b),\beta)\cong \HH_q(M,\beta-b)(b).
\end{eqnarray}
We close this section by recording two important vanishing results for 
Euler--Koszul homology.

\begin{proposition}\label{prop-HH qis 0}
For a weakly toric module $M$, the following are equivalent:
\begin{enumerate}
\item $\HH_i(M,\beta)=0$ for all $i\geq0$,
\item $\HH_0(M,\beta)=0$,
\item $\beta\notin\qdeg(M).$
\end{enumerate}
\end{proposition}
\begin{proof} See \cite[Theorem~5.4]{EKDI}.
\end{proof}

\begin{theorem}\label{thm-CM HH rel}
Let $M$ be a weakly toric module.
Then $\HH_i(M,\beta)=0$ for all $i>0$ and for all $\beta\in\C^d$
if and only if $M$ is a maximal Cohen--Macaulay $S_A$-module.
\end{theorem}
\begin{proof}
See \cite[Theorem~6.6]{MMW} for the toric case.
The extension to the weakly toric case can be found in \cite{EKDI}.
\end{proof}

\section{Euler--Koszul homology and toric face modules}
\label{sec:EKstr}

Theorem~\ref{thm-CM HH rel} provides a criterion for higher vanishing of 
Euler--Koszul homology via maximal Cohen--Macaulay $S_A$-modules. 
In this section, we provide a description of the Euler--Koszul homology modules 
of maximal Cohen--Macaulay $S_F$-modules for a face $F\preceq A$ 
and use it to understand the images of maps between such modules. 

Throughout this section, 
$N$ is a toric $S_F$-module for a face $F\preceq A$. 
Recall the definitions of toric, $S_F$, $D_F$, and $R_F$ from 
Definition~\ref{def:semigroup ring D}, and let 
\[
x_{F^c} = \{x_i\mid a_i \in F^c\}.
\] 

\begin{notation}\label{not-HHF}
Let $\cI_F$ be the lexicographically first subset of $\{1,2,\dots,d\}$ 
of cardinality $\dim(F)$ such that $\{E_i-\beta_i\}_{i\in\cI_F}$ is a set of 
linearly independent Euler operators on $D\otimes_R N$. 
The existence of $\cI_F$ follows from the fact that the matrix $A$ has full rank. 
We use $\KK^F_\bullet(N,\beta)$ to denote the Euler--Koszul complex on 
$D_F \otimes_{R_F} N$ given by the operators $\{E_i - \beta_i\}_{i\in\cI_F}$, 
and set 
\[
\HH^F_i(N,\beta) = H_i(\KK^F_\bullet(N,\beta)).
\]
\end{notation}

Using the standard basis of $\Z A=\Z^d$, let
\[
\Z F^{\bot} = 
	\left\{ v\in\Z^d\ \Bigg\vert\ 
		\sum_{i=1}^d v_ia_{ij}=0\ \forall\ a_j\in F \right\},
\]
and let $\bigwedge^\bullet\Z F^\bot$ denote a complex 
with trivial differentials. 
We show now that when $\beta\in\C F$, $\KK_\bullet(N,\beta)$ 
is quasi-isomorphic to a complex involving $\KK^F_\bullet(N,\beta)$ 
and $\bigwedge^\bullet\Z F^\bot$.

\begin{proposition}\label{prop-str CM EK complex}
Let $F\preceq A$ and $N$ be a toric $S_F$-module.
If $\beta\in\CC F$, then there is a quasi-isomorphism of complexes
\begin{eqnarray}\label{12}
\KK_\bullet(N,\beta)
    \simeq_{qis} \C[x_{F^c}] \otimes_\C
    \KK_\bullet^F(N,\beta) \otimes_\Z
    \left(\textstyle\bigwedge^\bullet\Z F^\bot\right).
\end{eqnarray}
In particular, if $N$ is maximal Cohen--Macaulay as an $S_F$-module, 
there is a decomposition
\begin{eqnarray}\label{6}
\HH_\bullet(N,\beta)
    = \C[x_{F^c}] \otimes_\C \HH_0^F(N,\beta) \otimes_\Z
     \left(\textstyle\bigwedge^\bullet\Z F^\bot\right).
\end{eqnarray}
\end{proposition}

Under the hypotheses of Proposition~\ref{prop-str CM EK complex},
\[
\HH_i(N,\beta) = \C[x_{F^c}] \otimes_\C
    \HH_0^F(N,\beta^F)^{\binom{\codim(F)}{i}},
\]
for $i>0$, as shown in \cite{okuyama}.
In particular,
\begin{eqnarray}\label{3}
    \rank \HH_i(N,\beta) = \binom{\codim(F)}{i}
        \cdot \rank \HH_0(N,\beta).
\end{eqnarray}
We show in Proposition~\ref{prop-str/rk image} that surjections of 
maximal Cohen--Macaulay toric modules for nested faces yield induced maps 
on Euler--Koszul homology that respect the decompositions of \eqref{6}. 
The additional information stored in $\bigwedge^\bullet\Z F^\bot$ 
of \eqref{12} shows how images of collections 
of such surjections overlap, which will be vital to our 
calculation of $j(\beta)$ in Section~\ref{sec:comb of rank}.

\begin{proof}[Proof of Proposition~\ref{prop-str CM EK complex}]
\label{pf-gF}
Fix a matrix $g_F\in\text{GL}_d(\Z)$ such that the entries of 
each row of $g_F F$ 
not corresponding to $\cI_F$ are zero and 
the rows of $A$ that do correspond to $\cI_F$ are identical in 
$A$ and $g_F A$.
Setting $A'=g_F A$, $S_A$ and $S_{A'}$ are isomorphic rings, 
and the matrix $g_F$ gives a bijection of their degree sets, 
sending $\N A$ to $\N A'$.
This identification makes $N$ a $\Z A'$-graded $S_{F'}$-module, 
where $F' = g_F F$, 
and there is a quasi-isomorphism of complexes
\begin{eqnarray*}
    \KK_\bullet(N,\beta) \simeq_{qis} \KK^{A'}_\bullet(N,g_F \beta).
\end{eqnarray*}

Let $F'=g_F F$ and $\beta' = g_F \beta$, and recall that $A' = g_F A$.
By the definition of $g_F$, $\beta'_i = 0$ 
for $i\notin\cI_F$ because $\beta'\in\C F'$. 
Let $D_{A'}$ and $R_{A'}$ denote the Weyl algebra and the polynomial 
ring $\C[\del]$ with an $A'$-grading.
Since $N$ is an $S_{F'}$-module, 
$0= \del_{F'^c}\otimes N\subseteq D_{A'} \otimes_{R_{A'}} N$, 
and so there is an isomorphism
$D_{A'} \otimes_{R_{A'}} N \cong \C[x_{F'^c}] \otimes_\C D_{F'} \otimes_{R_{F'}} N$. 
Hence the action of each element in 
$\{\sum_{j=1}^n a'_{ij}x_j\del_j\}_{i\notin\cI_F}$ 
on $D_{A'} \otimes_{R_{A'}} N$ is 0.

If $\{e_1,\dots,e_d\}$ denotes the standard basis of $\Z^d=\Z A'$, 
then the set $\{g_F^{-1} e_i\}_{i\notin\cI_F}$ generates 
$\Z F^\bot$ by choice of $g_F$. 
Applying the isomorphism 
$D_F\otimes_{R_F}N\cong D_{F'}\otimes_{R_{F'}}N$ 
in the reverse direction, we obtain \eqref{12}.
Finally, if $N$ is maximal Cohen--Macaulay as an $S_F$-module, 
$\HH^F_i(N,\beta^F)=0$ for all $i>0$ by Theorem~\ref{thm-CM HH rel}.
\end{proof}

\begin{remark}\label{remark-diffl in CM EK complex}
Let $\delta$ and $\kappa_F$ respectively denote the differentials of the 
Euler--Koszul complexes 
$\KK^{g_F A}_\bullet(N,g_F \beta)$ and $\KK^F_\bullet(N,\beta)$. 
Under the hypotheses of Proposition~\ref{prop-str CM EK complex}, 
if $i+j=q$ and 
$f \otimes a\otimes b \in
    \C[x_{F^c}] \otimes_\C \KK^F_i(N,\beta) \otimes_\Z
    \left(\textstyle\bigwedge^j\Z F^\bot\right)$,
then
\[
\delta(f\otimes a\otimes b) = f \otimes \kappa_F(a) \otimes b
\]
is an element of
$ \C[x_{F^c}] \otimes_\C \KK^F_{i-1}(N,\beta) \otimes_\C
    \left(\textstyle\bigwedge^j\Z F^\bot\right) \subseteq \KK^{A'}_{q-1}(N,\beta).$
\end{remark}

\begin{example}\label{example-EK CM comp}
Consider the matrix
\[
A=\left[\begin{matrix}
    1 & 1 & 1 & 1 \\
    0 & 1 & 0 & 1 \\
    0 & 0 & 1 & 1 \end{matrix}\right].
\]
Set $\beta =
\left[\begin{smallmatrix}0\\0\\0\end{smallmatrix}\right] \in\C^3$,
and let $e_1,e_2,e_3$ denote the standard basis vectors in $\Z A~=~\Z^3$.
Notice that every face ring of $A$ is Cohen--Macaulay because the
semigroup generated by each face of $A$ is saturated.
For the face $\nothing\preceq A$, we choose 
$g_\nothing$ to be the identity matrix. 
The proof of Proposition~\ref{prop-str CM EK complex}
shows that there is an isomorphism of complexes
$\KK_\bullet(S_\nothing,\beta) \cong
    \bigotimes_{i=1}^3 \left(\xymatrix{ \C[x]\cdot e_i \ar[r]^{0\cdot} & \C[x]}\right),$
so the Euler--Koszul homology of $S_\nothing$ at $\beta$ is
\[
\HH_\bullet(S_\nothing,\beta) =
    \textstyle\bigwedge^\bullet\left(\bigoplus_{i=1}^3 \C[x] \cdot e_i\right).
\]
For the face $F=[a_1\ a_2]$ of $A$, $(A,\beta)$ is again already 
in the desired form, so take $g_F$ to be the identity matrix and write  
$\C[x]\<\del_1,\del_2\>$ in place of
$\C[x_{F^c}] \otimes_\C D_F \otimes_{R_F} S_F$. 
Then Proposition~\ref{prop-str CM EK complex} implies that
\begin{align*}
\HH_q(S_{F},\beta) &=
    \begin{cases}
        \C[x]\<\del_1,\del_2\> & \text{if $q=0$,}\\
        \C[x]\<\del_1,\del_2\>\cdot e_3 & \text{if $q=1$,}\\
        0 & \text{otherwise.}
    \end{cases}
\intertext{For $G_1=[a_1]$ and $g_{G_1}$ as the identity matrix,}
\HH_q(S_{G_1},\beta) &=
    \begin{cases}
        \C[x]\<\del_1\> & \text{if $q=0$,}\\
        \C[x]\<\del_1\>\cdot e_2 \oplus \C[x]\<\del_1\>\cdot e_3 & \text{if $q=1$,}\\
        \C[x]\<\del_1\>\cdot e_2\wedge e_3 & \text{if $q=2$,}\\
        0 & \text{otherwise.}
    \end{cases}
\intertext{For the face $G_2 = [a_2]$, setting 
$g_{G_2} = \left[\begin{smallmatrix}
        1&\phantom{-}0&0\\
        1&-1&0\\
        0&\phantom{-}0&1 \end{smallmatrix}\right]$
yields the decomposition}
\HH_q(S_{G_2},\beta) &=
    \begin{cases}
        \C[x]\<\del_2\> & \text{if $q=0$,}\\
        \C[x]\<\del_2\>\cdot (e_1-e_2) \oplus \C[x]\<\del_1\>\cdot e_3 & \text{if $q=1$,}\\
        \C[x]\<\del_2\>\cdot (e_1-e_2)\wedge e_3 & \text{if $q=2$,}\\
        0 & \text{otherwise.}
    \end{cases}
\end{align*}
\end{example}

\begin{lemma}\label{lemma-EK morph}
Let $G\preceq F$ be faces of $A$,
$N$ be a toric $S_F$-module, and $L$ be a toric $S_G$-module.
Regard $N$ and $L$ as toric $S_A$-modules via the natural maps
$S_A\twoheadrightarrow S_F\twoheadrightarrow S_G$.
Let $\pi:N\rightarrow L$ be a morphism of $S_A$-modules.
Then there is a commutative diagram 
\begin{eqnarray*}
\xymatrix{
\KK_\bullet(N,\beta) \ar[r]^{\KK_\bullet(\pi,\beta)} \ar[d] &
\KK_\bullet(L,\beta) \ar[d] \\
\C[x_{F^c}] \otimes_\C
\KK^{F}_\bullet(N,\beta) \otimes_\Z
\left( \textstyle\bigwedge^\bullet \Z F^\bot \right) \ar[r] &
\C[x_{G^c}] \otimes_\C
\KK^{G}_\bullet(L,\beta) \otimes_\Z
\left( \textstyle\bigwedge^\bullet \Z G^\bot \right)
}
\end{eqnarray*}
with vertical maps as in \eqref{12}.
\end{lemma}
\begin{proof}
By choice of $\cI_F$ and $\cI_G$ in Notation~\ref{not-HHF}
and the corresponding $g_F$ and $g_G$, the diagram
\[
\xymatrix{
\KK_\bullet(N,\beta) \ar[rr]^{\KK_\bullet(\pi,\beta)} \ar[d] &&
\KK_\bullet(L,\beta) \ar[d] \\
\KK^{g_F A}_\bullet(N,\beta) \ar[r] & 
\KK^{g_F A}_\bullet(L,\beta) \ar[r] & \KK^{g_G A}_\bullet(L,\beta)
}
\]
commutes. Hence the result follows from the proof of 
Proposition~\ref{prop-str CM EK complex}. 
\end{proof}

\begin{proposition}\label{prop-str/rk image}
Let $G\preceq F$ be faces of $A$,
$N$ be a maximal Cohen--Macaulay toric $S_F$-module,
and $L$ be a maximal Cohen--Macaulay toric $S_G$-module.
Regard $N$ and $L$ as toric $S_A$-modules via the natural maps
$S_A\twoheadrightarrow S_F\twoheadrightarrow S_G$.
Let $\pi:N\twoheadrightarrow L$ be a surjection of $S_A$-modules.
If $\beta\in\C G$, then
\[
\image \HH_\bullet(\pi,\beta) = \C[x_{G^c}]
    \otimes_\C \HH_0^G(L,\beta) \otimes_\C
    \left(\textstyle\bigwedge^\bullet\Z F^\bot\right)
\]
as a submodule of
$\C[x_{G^c}] \otimes_\C \HH_0^G(L,\beta) \otimes_\C
    \left(\textstyle\bigwedge^\bullet\Z G^\bot\right)$.
\end{proposition}

\begin{example}
(Continuation of Example~\ref{example-EK CM comp})
\label{example-EK CM comp im} 
The surjection of face rings given by 
$\pi:S_{F}\twoheadrightarrow S_{G_1}$
induces the following image in Euler--Koszul homology:
\[
\image\HH_q(\pi,\beta) =
    \begin{cases}
    \C[x]\<\del_1\> & \text{if $q=0$,}\\
    \C[x]\<\del_1\>\cdot e_3 & \text{if $q=1$,}\\
    0 & \text{otherwise.}
    \end{cases} 
\]
\end{example}

\begin{proof}[Proof of Proposition~\ref{prop-str/rk image}]
With $A' = g_F A$, the image of $\HH_\bullet(\pi,\beta)$ is
isomorphic to the image of $\HH^{A'}_\bullet(\pi,\beta)$.
By Proposition~\ref{prop-str CM EK complex}, there are decompositions
\begin{align*}
\HH^{A'}_\bullet(N,\beta) &=
    \C[x_{F^c}] \otimes_\C \HH^F_0(N,\beta) \otimes_\Z
    \left(\textstyle\bigwedge^\bullet \Z F^\bot \right)
\intertext{and}
\HH^{A'}_\bullet(L,\beta) &=
    \C[x_{F^c}] \otimes_\C \HH^F_\bullet(L,\beta) \otimes_\Z
    \left(\textstyle\bigwedge^\bullet \Z F^\bot \right),
\end{align*}
so it is enough to find the image of $\HH^F_0(N,\beta)$ 
as a submodule of $\HH^F_\bullet(L,\beta)$. 
The result now follows because the sequence
\[
\xymatrix{
    \HH_0^F(N,\beta) \ar[rr]^{\HH_0^F(\pi,\beta)} &&
    \HH_0^F(L,\beta) \ar[r] & 0}
\]
is exact, 
$\HH_0^F(L,\beta) = \C[x_{F\setminus G}] \otimes_\C \HH^G_0(L,\beta)$, 
and $\HH^F_i(N,\beta)=0$ for $i>0$.
\end{proof}

\begin{example}(Continuation of Example~\ref{example-EK CM comp im})
\label{example-EK CM comp im intersect} 
Let $\pi_i:S_{G_i}\twoheadrightarrow S_\nothing$ for $i=1,2$. Then
\begin{align*}
\image \HH_q(\pi_1,\beta) &=
    \begin{cases}
        \C[x] & \text{if $q=0$,}\\
        \C[x]\cdot e_2 \oplus \C[x]\cdot e_3 & \text{if $q=1$,}\\
        \C[x]\cdot e_2\wedge e_3 & \text{if $q=2$,}\\
        0 & \text{otherwise,}
    \end{cases}
\intertext{and}
\image \HH_q(\pi_2,\beta) &=
    \begin{cases}
        \C[x] & \text{if $q=0$,}\\
        \C[x]\cdot (e_2-e_1) \oplus \C[x]\cdot e_3 & \text{if $q=1$,}\\
        \C[x]\cdot (e_2-e_1) \wedge e_3 & \text{if $q=2$,}\\
        0 & \text{otherwise.}
    \end{cases}
\end{align*}
The intersection of the images of Euler--Koszul homology at $\beta$ 
applied to $\pi_1$ and $\pi_2$ is
\[
[\image \HH_q(\pi_1,\beta)]\cap[\image \HH_q(\pi_2,\beta)] =
    \begin{cases}
    \C[x] & \text{if $q=0$,}\\
    \C[x]\cdot e_3 & \text{if $q=1$,}\\
    0 & \text{otherwise}
    \end{cases}
\]
because $\Z G_1^\bot \cap \Z G_2^\bot = \Z\cdot e_3$. 
\end{example}

We close this section with an observation that is vital to our rank jumps computations. 
For faces $F_1,F_2\preceq A$, set $G=F_1\cap F_2$.
Let $N_i$ be maximal Cohen--Macaulay toric $S_{F_i}$-modules,
$L$ be a maximal Cohen--Macaulay toric $S_G$-module, and
$\pi_i: N_i \twoheadrightarrow L$ be $S_A$-module surjections. 
Suppose that $\beta\in\C G$. 
Using the equality $\Z F^\bot\cap \Z G^\bot=\Z [F\cup G]^\bot$, 
Proposition~\ref{prop-str/rk image} implies that 
\[
\begin{array}{c}
\image \HH_i(\pi_1,\beta) \cap \image\HH_i(\pi_2,\beta) 
    =  \C[x_{G^c}] \otimes_\C \HH_0^G(L,\beta) \otimes_\C
    \left(\textstyle\bigwedge^i \Z [F\cup G]^\bot\right),
\end{array}
\]
which has rank
\begin{eqnarray*}
\binom{\codim_{\C^d}(\C F_1 + \C F_2)}{i} \cdot \rank \HH_0^G(L,\beta). 
\end{eqnarray*}

\section{Stratifications of the exceptional arrangement}
\label{sec:exceptional arrangement}

Let $\dM \subseteq \Z^d$ be a nonempty $\N A$--monoid 
(see Definition~\ref{def:monoid}), so that the nontrivial module 
$M = \C\{\dM\} \subseteq \C\{\Z^d\} \cong S_A[\del_A^{-1}]$ 
is weakly toric (see Example~\ref{ex-M wt}) and 
\begin{align}\label{eq-dM}
\dM = \deg(M). 
\end{align}
Since $\dM$ is an $\N A$-monoid, the generic rank of $\HH_0(M,-)$ is $\vol(A)$. 
The {\it rank jump} of $M$ at $\beta$ is the nonnegative integer
\begin{align*}
j(\beta)&=\rank\HH_0(M,\beta)-\vol(A), 
\intertext{
and the {\it exceptional arrangement} associated to $M$ is the set
}
\EE_A(M)&=\{\beta\in\CC^d\mid j(\beta)>0\}
\end{align*}
of parameters with nonzero rank jump. 
By \cite{MMW} and \cite{EKDI}, the exceptional arrangement 
can be described in terms of certain 
$\ext$ modules involving $M$, namely
\begin{eqnarray}\label{13}
\EE_A(M) = 
	- \bigcup_{i=0}^{d-1} \qdeg\left( \ext_{R}^{n-i}(M,R)( - \varepsilon_A) \right),
\end{eqnarray} 
where $\varepsilon_A = \sum_{i=1}^n a_i$.
It follows that $\EE_A(M)$ is a union of translates
of linear subspaces spanned by the faces of $A$, see \cite[Corollary~9.3]{MMW}.

We begin our study of $j(\beta)$ with the short exact sequence 
\begin{eqnarray}\label{1}
    0 \rightarrow M \rightarrow
    S_A[\del_A^{-1}] \rightarrow Q \rightarrow 0.
\end{eqnarray}
While $Q$ is not a Noetherian $S_A$-module, it is a filtered limit of 
Noetherian $\Z^d$-graded $S_A$-modules and is therefore weakly toric 
(see Definition~\ref{def-wtoric}). 
Thus the {\it ranking arrangement} of $M$
\[
\cR_A(M) = \qdeg(Q)
\]  
is an infinite union of translates of linear subspaces of $\C^d$ 
spanned by proper faces of $A$.
Since $S_A[\del_A^{-1}]$ is a maximal Cohen--Macaulay 
$S_A$-module, Theorem~\ref{thm-CM HH rel} 
implies that $\HH_i(S_A[\del_A^{-1}],\beta)=0$ for all $i>0$.
Moreover, by \cite[Theorem~4.2]{okuyama},  
$\rank \HH_0(S_A[\del_A^{-1}],\beta)=\vol(A)$ for all $\beta$.
Examination of the long exact sequence in Euler--Koszul 
homology from \eqref{1} reveals that
\begin{equation}\label{2}
    j(\beta)=
    \rank\HH_1(Q,\beta)-\rank\HH_0(Q,\beta).
\end{equation}
This implies that for $\beta\in\EE_A(M)$, 
$\HH_1(Q,\beta)$ is nonzero.
Therefore there is an inclusion $\EE_A(M)\subseteq\cR_A(M)$.
We make this relationship precise in Theorem~\ref{thm-str comparison}.

\begin{lemma}\label{lemma-finite inside}
Let $v\in\Z^d$. The number of irreducible components of
$\cR_A(S_A)$ which intersect $v+\RR_{\geq0}A$ is finite.
\end{lemma}
\begin{proof}
View $S_A$ and its shifted saturation $\wt{S}_A(-v)$ 
as graded submodules of $S_A[\del_A^{-1}]$.
To see that the intersection $\cR_A(S_A)\cap(v+\R_{\geq0} A)$ 
involves only a finite number of irreducible components 
of $\cR_A(S_A)$, it is enough to show that the arrangement 
given by the quasidegrees of the module 
$\wt{S}_A(-v)/(S_A\cap\wt{S}_A(-v))$ has finitely many 
irreducible components. 
This follows since $\wt{S}_A(-v)$ is toric.
\end{proof}

Recall from \eqref{eq-dM} that $\dM = \deg(M)$. 
For $b\in\Z^d$ and $F\preceq A$, let 
$\nabla(M,b)=\{F\preceq A\mid b\in\dM+\Z F\}$. 

\begin{lemma}\label{lemma-MM05}
Let $M\subseteq S_A[\del_A^{-1}]$ be a weakly toric module, $b\in\Z^d$, 
$F\preceq A$ be a face of codimension at least two, 
and $\alpha\in \Z^d$ be an interior vector of $\N F$.
If $F$ is maximal among faces of $A$ not in $\nabla(M,b)$, 
then for all sufficiently large positive integers $m$, the vector
$b-m\alpha \in \EE_A(M)$ is an exceptional degree of $M$.
\end{lemma}
\begin{proof}
This is \cite[Lemma~14]{MM05} when $M=S_A$ and $A$ is homogeneous. 
(The matrix $A$ is called {\it homogeneous} when the vector
$(1,1,\dots,1)$ is in the $\Q$-row span of $A$.) 
The same argument yields this generalization 
by $\Z^d$-graded local duality, see \cite[Section~3.5]{bruns-herzog}.
\end{proof}

\begin{theorem}\label{thm-str comparison}
Let $M\subseteq S_A[\del_A^{-1}]$ be a weakly toric module.
The ranking arrangement $\cR_A(M)$ contains the 
exceptional arrangement $\EE_A(M)$ and
\begin{eqnarray*}
\cR_A(M)=\EE_A(M)\cup\cZ_A(M),
\end{eqnarray*}
where $\cZ_A(M)$ is pure of codimension 1.
\end{theorem}
\begin{proof} 
We must show that $\EE_A(M)$ contains all irreducible components 
of $\cR_A(M)$ of codimension at least two.
To this end, let $\beta\in\cR_A(M)$ be such that 
$\beta+\C F\subseteq \cR_A(M) = \qdeg(Q)$ 
is an irreducible component with $\codim(F)\geq 2$.
Then there are submodules $M',M''\subseteq Q$ and $b'\in\Z^d$ such that 
$M'/M''\cong S_F(b')$ and $b'+\C F = \beta+\C F$. 
In fact, there is a $b\in\deg(M'/M'')$ with $b + \wt{\N F}\subseteq\deg(Q)$ and 
$b+\C F = \beta+\C F$.
We may further choose $b$ so that $F$ is maximal among faces
of $A$ that are not in the set $\nabla(M,b)$.
To see this, first note that $F\notin\nabla(M,b + r)$ for all $r\in\wt{\N F}$. 
Indeed, for if $(b + \wt{\N F})\cap(\dM+\Z F)\neq\nothing$, then there are 
$a\in b + \wt{\N F}$ and $s\in\N F$ with 
$a + s \in \dM \cap (b + \wt{\N F}) \subseteq \dM \cap \deg(Q) = \nothing$, 
a contradiction.

Since $b+\C F = \beta+\C F$ is an irreducible component of 
$\qdeg(Q)$, it suffices to show that $b$ can be chosen 
so that each facet $F'$ of $A$ is in $\nabla(M,b)$.
First, if $F\npreceq F'$, then by Lemma~\ref{lemma-finite inside}, 
there are at most a finite number of translates of $\C F'$ that 
define components of $\qdeg(Q)$ and intersect 
$b+\wt{\N F}$; write these as $c_1+\C F',\dots,c_k+\C F'$.
If necessary, replace $b$ by a vector $b'\in b+\wt{\N F}$ such that 
$(b'+\wt{\N F})\cap (c_i+\C F')=\nothing$ to assume 
that $F'$ is in $\nabla(M,b)$. 
Note that after such a replacement, it is still true that 
$(b+\wt{\N F})\cap(\dM+\Z F)\neq\nothing$ by the previous paragraph, 
so $F\notin\nabla(M,b)$. 
Next, suppose that $F\preceq F'$.
If $(b+\wt{\N F'})\cap \dM=\nothing$, then 
$b+\wt{\N F'}\subseteq\deg(Q)$,
an impossibility because $b+\C F$ defines an 
irreducible component of $\qdeg(Q)$.
Thus it must be that $(b+\wt{\N F'})\cap \dM\neq\nothing$. 
In this case, $b\in\dM+\Z F'$, so $F'$ is in $\nabla(M,b)$.
Hence every facet $F'$ of $A$ is in $\nabla(M,b)$,
and the claim on the choice of $b$ is established.

Let $\alpha\in\Z^d$ be an interior vector of $\N F$.
Lemma \ref{lemma-MM05} implies that for all sufficiently large integers $m$, 
the vector $b-m\alpha\in\EE_A(M)$. 
Therefore $\beta+\C F = b+\C F\subseteq\EE_A(M)$.
\end{proof}

\begin{notation}\label{not-beta cmpt}
For $\beta\in\C^d$, the \emph{$\beta$-components} 
$\cR_A(M,\beta)$ of the ranking arrangement of $M$ are the union of 
the irreducible components of $\cR_A(M)$ which contain $\beta$.
Since $A$ has a finite number of faces, $\cR_A(M,\beta)$ 
has finitely many irreducible components. 
\end{notation}

By \cite[Porism~9.5]{MMW}, the exceptional arrangement 
$\EE_A(S_A)$ of the $A$-hypergeo\-metric system 
$\HH_0(S_A,\beta) = M_A(\beta)$ 
has codimension at least two. 
In the following example we show that there may be components of 
$\EE_A(S_A)$ which are embedded in 
codimension 1 components of the ranking arrangement $\cR_A(S_A)$. 
In particular, the Zariski closure of $\EE_A(S_A)\setminus\cZ_A(S_A)$ 
may not agree with $\EE_A(S_A)$.

\begin{example}\label{example-hidden}
Let 
\[
A=\left[ \begin{matrix} 
	2 & 3 & 0 & 0 & 1 & 0 & 1\\
	0 & 1 & 2 & 0 & 0 & 1 & 1\\
	0 & 0 & 0 & 1 & 1 & 1 & 1
	\end{matrix}\right]
\]
with $\vol(A) = 15$, 
and label the faces $F=[ a_1\ a_2\ a_3 ]$ and $G = [ a_3 ]$. 
With $\beta = \left[\begin{smallmatrix} 1\\0\\0\end{smallmatrix}\right]$,  
the exceptional arrangement of $M=S_A$ 
is properly contained in a hyperplane component 
of the ranking arrangement: 
\[
\EE_A(S_A)
=\beta + \CC G
\ \subsetneq \
\beta + \CC F =\cR_A(S_A,\beta)
\ \subsetneq  \ \cR_A(S_A).
\]
We will discuss the rank jumps of $M_A(\beta)$ in 
Examples~\ref{example-hidden2} and~\ref{example-hidden3}. 
\end{example}
\smallskip

One goal of Section~\ref{sec:comb of rank} is to understand 
the structure of the sets 
$\EE_A^i(M)=\{\beta\in\CC^d\mid j(\beta)>i\}$. 
We will achieve this by stratifying $\EE_A(M)$ by ranking slabs 
(see Definition~\ref{def-ranking slab}).
Another description of ranking slabs 
(via translates of certain lattices contained in the $\beta$-components 
$\cR_A(M,\beta)$) will be given in Proposition~\ref{prop-zc E}. 

\begin{definition}
Let $X$ and $Y$ be subspace arrangements in $\C^d$. 
We say that a stratification $\cS$ of $X$ \emph{respects} $Y$ 
if for each irreducible component $Z$ of $Y$ and each stratum $S\in \cS$, 
either $S\cap Z = \nothing$ or $S\subseteq Z$. 
\end{definition}

\begin{definition}\label{def-ranking slab}
A \emph{ranking slab} of $M$ is 
a stratum in the coarsest stratification of $\EE_A(M)$ 
that respects each of the following arrangements: 
$\cR_A(M)$ 
and $- \qdeg\left( \ext_{R}^{n-i}(M,R)( - \varepsilon_A) \right)$ 
for $0\leq i < d$.
\end{definition}

Since each of the arrangements used in Definition~\ref{def-ranking slab} 
is determined by the quasidegrees of a weakly toric module, 
the closure of each ranking slab of $M$ is the translate of a 
linear subspace of $\C^d$ that is generated by a face of $A$. 
Corollary~\ref{cor-strat vs equiv} states 
that $j(-)$ is constant on each ranking slab, so 
each $\EE_A^i(M)$ with $i\geq0$ 
is a union of translates of linear subspaces of $\C^d$ 
that are spanned by faces of $A$.
It then follows that the stratification of $\EE_A(M)$ 
by ranking slabs refines its rank stratification. 
While this is generally a strict refinement, 
Examples~\ref{example-hidden3}~and~\ref{example-4diml hidden} 
show that the two stratifications may coincide 
for parameters close enough to the cone $\R_{\geq0}A$.
We wish to emphasize that the rank jump $j(\beta)$ 
is not simply determined by holes within the semigroup $\N A$, 
as can be seen in Example~\ref{ex-first example}. 

\begin{definition}\label{def:slab}
A {\it slab} is a set of parameters in $\C^d$ that lie on a 
unique irreducible component 
of the exceptional arrangement $\EE_A(M)$ \cite{MM05}. 
\end{definition}

We will show by example that rank need not be constant on a slab. 
In Example~\ref{example-4diml hidden}, 
this failure results from ``embedded" components of $\EE_A(M)$, 
while in Example~\ref{example-nonconstant slab}, 
it is due to the hyperplanes of $\cR_A(M)$ that strictly 
refine the arrangement stratification of $\EE_A(M)$.  
Together with Example~\ref{plane-line-4diml}, 
these examples show that each of the 
arrangements listed in Definition~\ref{def-ranking slab} 
is necessary to determine such a geometric 
refinement of the rank stratification of $\EE_A(M)$.

\section{Ranking toric modules}
\label{sec:r-t mods}

As in Section~\ref{sec:exceptional arrangement}, let $\dM \subseteq \Z^d$ 
be a nonempty $\N A$--monoid (see Definition~\ref{def:monoid}), 
so that $M = \C\{\dM\} \subseteq \C\{\Z^d\} \cong S_A[\del_A^{-1}]$ 
is a nontrivial weakly toric module (see Example~\ref{ex-M wt}). 
For a fixed $\beta\in\EE_A(M)$, we know from \eqref{2} that 
$Q$ can be used to compute the rank jump $j(\beta)$. 
However, this module contains a large amount of excess 
information that does not play a role in $\HH_\bullet(Q;\beta)$. 
To isolate the graded pieces of $Q$ that impact $j(\beta)$, 
we will define weakly toric modules 
$S^{\beta}\subseteq T^{\beta}$ so that

\vspace{-3.5ex}
\begin{eqnarray}\label{remark-goal modules}
\begin{minipage}[t]{12cm} 
\begin{enumerate}
\item \label{prop1} $M \subseteq S^{\beta} \subseteq
        T^{\beta} \subseteq S_A[\del_A^{-1}],$
\item \label{prop2} $\cR_A(M,\beta) =
        \qdeg\left(\dfrac{T^{\beta}}{S^{\beta}}\right),$
\item \label{prop3} $\beta \notin
        \qdeg\left(\dfrac{S_A[\del_A^{-1}]}{T^{\beta}}\right)$,
\item \label{prop4} $\beta \notin
        \qdeg\left(\dfrac{S^{\beta}}{M}\right)$, and
\item \label{prop5} $\PP^{\beta} =
        \deg\left(\dfrac{T^{\beta}}{S^{\beta}}\right)$
        is a union of translates of $\wt{\N F}$ for various $F \preceq A$.
\end{enumerate}
\end{minipage}
\end{eqnarray}

In Proposition~\ref{prop-qis}, we show that 
Properties~\eqref{prop1}-\eqref{prop4} of \eqref{remark-goal modules} 
allow us to replace $Q$ with 
$P^{\beta}=T^{\beta}/S^{\beta}$ when
calculating $j(\beta)$. 
To use this module to actually compute $j(\beta)$, 
we will encounter other toric modules with structure similar 
to $P^{\beta}$, which are called \emph{ranking toric modules}. 
Property~\eqref{prop5} allows $P^{\beta}$ 
(and similarly, any ranking toric module) 
to be decomposed into \emph{simple ranking toric modules}. 
These modules are constructed so that their 
Euler--Koszul homology modules have easily computable ranks. 
At the end of Section~\ref{subsec:simple rtm}, we outline more specifically how 
simple ranking toric modules will play a role in our computation of $j(\beta)$. 

\subsection{Combinatorial objects controlling rank}
\label{subsec:r-t construction}

We now construct the class of {\it ranking toric modules}, which includes 
the module $P^{\beta}$ coming from \eqref{remark-goal modules}. 
These modules will be constructed via their degree sets, 
which are unions of $\wt{\N F}$--modules for various $F\preceq A$. 
We begin by isolating the translated lattices contained in $\deg(Q)$ that lie in 
the $\beta$-components $\cR_A(M,\beta) \subseteq \qdeg(Q) = \cR_A(M)$ of 
the ranking arrangement of $M$ (see Notation~\ref{not-beta cmpt}). 
The union of these translated lattices will be denoted by $\bbE^{\beta}$. 

\begin{definition}\label{def-EP} \ \
\begin{enumerate}
\item Let
    \[
    \FF(\beta)
        = \{ F \preceq A \mid \beta + \C F \subseteq \cR_A(M) \}
    \]
    be the set of faces of $A$ corresponding to the $\beta$-components 
    $\cR_A(M,\beta)$ of the ranking arrangement of $M$. 
    This set $\FF(\beta)$ is a polyhedral cell complex, 
    and $\cR_A(M,\beta)$ is the union 
    $\cR_A(M,\beta) = \bigcup_{F\in \FF(\beta)}(\beta+\C F).$
\item For each $F \in \FF(\beta)$, let
    \begin{eqnarray*}
    \bbE_F^{\beta}
        = \Z^d \cap (\beta+\C F)\setminus(\dM+\Z F).
    \end{eqnarray*}
   	Lemma~\ref{lemma-MM05} and Theorem~\ref{thm-str comparison} 
	together imply that $\bbE_F^{\beta}$ is nonempty 
	exactly when there is a containment
	$\beta+\C F\subseteq \cR_A(M,\beta)$. 
\item Since $\dM$ is an $\N A$--monoid, 
    $\dM+\Z F$ is closed under addition, 
    so $\bbE_F^{\beta}$ is $\Z F$-stable.
    Thus there is a finite set of $\Z F$-orbit representatives 
    $B_F^{\beta}$ such that 
    \begin{equation}\label{eqn-partition bbE}
        \bbE_F^{\beta}
            = \bigsqcup_{b\in B_F^{\beta}}(b+\Z F)
    \end{equation}
    is partitioned into $\Z F$-orbits as a disjoint union over $B_F^{\beta}$.
    Notice that $|B_F^{\beta}|\leq[(\Z^d\cap\R F):\Z F]$. 
\item The $\Z F$-orbits $b+\Z F$ in 
	\eqref{eqn-partition bbE} are the translated lattices 
	that we will use to construct ranking toric modules. 
	Each is determined by the pair $(F,b)$. 
	We denote the collection of such pairs by 
	\begin{eqnarray*}
		\cJ(\beta) = 
		\left\{ (F,b) \in \FF(\beta) \times B_F^{\beta} 
		\ \Big\vert\ (b+\Z F)\subseteq\bbE_F^{\beta} \right\}. 
	\end{eqnarray*}
\item For a subset $J \subseteq \cJ(\beta)$, let 
	\begin{align*}
	\bbE_J^{\beta} \ &= \ \bigcup_{(F,b)\in J} (b+\Z F). 
	\intertext{The maximal case determines the \emph{ranking lattices} of $M$ at $\beta$:}
	\bbE^{\beta} \ 
	:= \ \bbE_{\cJ(\beta)}^{\beta} \ 
	&= \ \bigcup_{(F,b)\in\cJ(\beta)} (b+\Z F). 
	\end{align*}
\end{enumerate}
\end{definition}

\begin{notation}\label{not-omitJ}
Many of the objects we define in this section are dependent upon 
a subset $J\subseteq \cJ(\beta)$, 
and this dependence is indicated by the subscript $J$. 
Whenever we omit this subscript, it is understood that $J = \cJ(\beta)$. 
\end{notation}

By the upcoming Proposition~\ref{prop-zc E}, two parameters 
$\beta,\beta'\in\C^d$ 
belong to the same ranking slab exactly when 
$\bbE^{\beta} = \bbE^{\beta'}$. 
This is what will be used to show that the rank jump $j(\beta)$) 
is constant on ranking slabs. 

\begin{lemma}\label{lemma-zc E addition}
The Zariski closure of the ranking lattices $\bbE^{\beta}$ of $M$ at $\beta$ 
coincides with the $\beta$-components $\cR_A(M,\beta)$ 
of the ranking arrangement. 
\end{lemma}
\begin{proof}
It is clear from the definitions that 
$\bbE^{\beta}\subseteq \cR_A(M,\beta)$. 
For the reverse containment, 
if $\beta + \C F \subseteq \cR_A(M,\beta)$, then there 
exists a vector 
$b\in \beta + \C F$ such that 
$b + \wt{\N F} \subseteq \Z A \setminus \dM$. 
This implies that $(b + \N F) \cap (\dM + \Z F)$ is empty, 
so the claim now follows from the definition of quasidegree sets 
in Definitions~\ref{def-qdeg}~and~\ref{def-wtoric}. 
\end{proof}

\begin{proposition}\label{prop-zc E}
The parameters $\beta,\beta'\in\C^d$ belong to the same ranking slab 
if and only if the ranking lattices of $M$ at $\beta$ and $\beta'$ coincide, 
that is, if $\bbE^{\beta} = \bbE^{\beta'}.$
\end{proposition}
\begin{proof}
This is a consequence of Lemma~\ref{lemma-zc E addition}, 
Lemma~\ref{lemma-MM05}, and Theorem~\ref{thm-str comparison}. 
\end{proof}

\begin{notation}
In light of Proposition~\ref{prop-zc E}, use equality of ranking lattices to extend 
the ranking slab stratification of $\EE_A(M)$ to the parameter space $\C^d$. 
\end{notation}

One might try making the sets $\bbE_J^{\beta}$ 
in Definition~\ref{def-EP}
the degree sets of ranking toric modules. 
However, while the natural map  
$\bbE_F^{\beta} \rightarrow \bbE_G^{\beta}$ 
given by faces $G\preceq F$ 
induces a vector space map 
$\C\{\bbE_F^{\beta}\}\rightarrow\C\{\bbE_G^{\beta}\}$, 
this induced map is not a morphism of $S_F$-modules because 
it sends units to zero. 
To overcome this, we introduce the lattice points in a certain polyhedron, 
denoted by $\cC_A(\beta)$, and intersect it 
with $\bbE_J^\beta$ to produce the degree set of a ranking toric module.

\begin{definition}\label{def-CP} \ \ 
\begin{enumerate}
\item Recall the primitive integral support functions $p_F$ from the beginning of Section~\ref{sec:defns}. In order to construct a ranking toric module from 
	$\bbE_J^{\beta}$, (and achieve the various quasidegree sets 
	proposed in \eqref{remark-goal modules}), set
    \[ \phantom{xxx}
    \cC_A(\beta) =
        \left\{
        v\in\Z^d\ \bigg\vert\ \text{ for each facet $F$ of $A$,
        }
            \begin{array}{ll}
            p_F(v)\geq p_F(\beta) &
                \text{if $p_F(\beta)\in\R$,} \\
            p_F(v)\geq 0 & \text{else.}
            \end{array}
        \right\}.
    \]
	For $\beta\in\R^d$, notice that 
	$\cC_A(\beta)=\Z^d\cap(\beta+\RR_{\geq 0}A)$ 
	is simply the integral points in the cone $\RR_{\geq0}A$ 
	after translation by $\beta$. 
\item For a pair $(F,b)\in\cJ(\beta)$, let  
    \begin{align}
    \PP_{F,b}^{\beta} & \phantom{:}= \cC_A(\beta) \cap [b+\Z F]. \nonumber
    \intertext{The degree sets of ranking toric modules are of the form}
	\label{17}
    \PP_J^{\beta} & \phantom{:}=  \bigcup_{(F,b)\in J} \PP_{F,b}^{\beta} 
    \ = \ \cC_A(\beta) \cap \bbE_J^{\beta}
	\intertext{for $J\subseteq\cJ(\beta)$. The largest of these is}
	\label{eqn-P}
	\PP^{\beta} \ &:= \ \PP_{\cJ(\beta)}^{\beta} \ 
	= \ \cC_A(\beta) \cap \bbE^{\beta}, 
	\end{align}
	the degree set appearing in \eqref{remark-goal modules}. 
\end{enumerate}
\end{definition}

\begin{example}
(Continuation of Example~\ref{example-hidden})
\label{example-hidden2}
With $\beta = 
\left[\begin{smallmatrix} 1 \\ 0 \\ 0 \end{smallmatrix}\right] \in\EE_A(S_A)$ 
and $b =  \left[\begin{smallmatrix} 1 \\ 1 \\ 0 \end{smallmatrix}\right]$, 
the sets of Definitions~\ref{def-EP}~and~\ref{def-TS and r-t} are 
\begin{eqnarray*}
\FF(\beta) &=& \{ \nothing, G, F \}, \\
\cJ(\beta) &=& \{ (\nothing,\beta), (G,b), (F,\beta) \}, \\
\bbE^{\beta} &=& [\beta + \Z F] \sqcup [b + \Z G], \\
 \text{and } \phantom{xx}
 \PP^{\beta} &=& [\beta + \N F] \sqcup [b + \N G]. 
\end{eqnarray*} 
\end{example}

Having defined the degree sets of ranking toric modules in \eqref{17},  
we now construct the modules themselves.   
Along the way, we meet the modules 
that satisfy the requirements of \eqref{remark-goal modules}.

\begin{definition}\label{def-TS and r-t} \ \ 
\begin{enumerate}
\item Each ranking toric module will be a quotient of the module $T^{\beta}$ 
	(for some $\beta\in\C^d$), where 
	\[
	\TT^{\beta} = 
		\dM\cup\left[\bigcup_{b\in\PP^{\beta}} (b+\wt{\N A})\right]
	\phantom{xx} \text{ and } \phantom{xx}
	T^{\beta}\ =\ \C\{\TT^{\beta}\}.
	\]	
	Notice that if $\beta \in \Z^d$, 
	then $\TT^{\beta} = \dM \cup (\beta+\wt{\N A})$. 
	The simplest case occurs when $\dM = \N A$ and $\beta \in \wt{\N A}$,   
	so that $\TT^{\beta} = \wt{\N A}$. 
\item For $J\subseteq\cJ(\beta)$, let 
	\[
	\SS_J^{\beta}  = \TT^{\beta}\setminus \PP_J^{\beta}
	\phantom{xx} \text{ and } \phantom{xx}
	S_J^{\beta}\ =\ \CC\{\SS_J^{\beta}\}.
	\]
    We show in Proposition~\ref{prop-basics R-toric} 
    that $T^{\beta}$ and 
    $S_J^{\beta}$ are indeed weakly toric modules (see Definition~\ref{def-wtoric}).
    When $J = \cJ(\beta)$, these modules satisfy the properties 
    \eqref{remark-goal modules}. By Notation~\ref{not-omitJ},  
    $S^{\beta} = S_{\cJ(\beta)}^{\beta}$. 
\item For a subset $J\subseteq\cJ(\beta)$, the quotient
    \begin{eqnarray*}
    P_J^{\beta} & = & \dfrac{T^{\beta}}{S_J^{\beta}}
    \end{eqnarray*}
    has degree set $\deg(P_J^{\beta})=\PP_J^{\beta}$, 
    as recorded in Proposition~\ref{prop-basics R-toric}. 
    In Proposition~\ref{prop-qis}, we show that $Q$ in \eqref{2} 
    can be replaced by 
    $P^{\beta} = P_{\cJ(\beta)}^{\beta}$
    when computing $j(\beta)$.
\item If a toric module $N$ is isomorphic to $P_J^{\beta}$ for a pair 
    $(M,\beta)$ and a subset $J\subseteq\cJ(\beta)$, we say that 
    $N$ is a {\it ranking toric module} determined by $J$.
\end{enumerate}
\end{definition}

\begin{proposition} \label{prop-basics R-toric}
Let $J \subseteq \cJ(\beta)$.
    There are containments of weakly toric modules:
    \begin{eqnarray}\label{26}
    M\subseteq S_J^{\beta} \subseteq 
    	T^{\beta}\subseteq S_A[\del_A^{-1}].
    \end{eqnarray}
    In particular, 
    $P_J^{\beta} = T^{\beta}/S_J^{\beta}$ is a 
    ranking toric module with degree set 
    $\PP_J^{\beta}$.
\end{proposition}
\begin{proof}
By construction, we have the containment 
$\PP_J^{\beta}\subseteq \bbE^{\beta}$.
Since the intersection of $\bbE^{\beta}$ and $\dM$ is empty,
$\PP_J^{\beta} \cap \dM$ is empty as well. 
Hence $M \subseteq S_J^{\beta}$. 
The other containments in \eqref{26} 
are obvious.
It is clear from the definitions that the degree sets of all 
modules in question are closed under addition with 
elements of $\N A$, so they are all weakly toric modules. 
For the second statement, 
since $\PP_J^{\beta} \subseteq \cC_A(\beta)$, 
$P_J^{\beta}$ is a finitely generated $S_A$-module 
and therefore $P_J^{\beta}$ is toric. 
\end{proof}

By Lemma~\ref{lemma-zc E addition} and the definition 
\eqref{eqn-P} of $\PP^{\beta}$, the arrangement 
$\qdeg(P^{\beta})$ coincides with the $\beta$-components 
$\cR_A(M,\beta)$ of $M$ at $\beta$. 
Further, the construction of  
$P^{\beta}$ in Definition~\ref{def-TS and r-t}
is such that $P^{\beta}$ can replace $Q$ 
in \eqref{2} when calculating $j(\beta)$.

\begin{proposition}\label{prop-qis}
The Euler--Koszul complexes $\KK_\bullet(Q,\beta)$ and
$\KK_\bullet(P^{\beta},\beta)$ are quasi-isomorphic.
In particular,
\begin{eqnarray}\label{5}
    j(\beta) =
    \rank \HH_1(P^{\beta},\beta) -
    \rank \HH_0(P^{\beta},\beta). 
\end{eqnarray}
\end{proposition}
\begin{proof}
Consider the short exact sequences
\begin{eqnarray*}
0 \rightarrow \dfrac{T^{\beta}}{M} \rightarrow
\dfrac{S_A[\del_A^{-1}]}{M} \rightarrow
\dfrac{S_A[\del_A^{-1}]}{T^{\beta}} \rightarrow 0
\quad \text{and} \quad
0 \rightarrow \dfrac{S^{\beta}}{M} \rightarrow
\dfrac{T^{\beta}}{M} \rightarrow
\dfrac{T^{\beta}}{S^{\beta}} \rightarrow 0.
\end{eqnarray*}
The definition of $\cC_A(\beta)$ ensures that $\beta$ is not a quasidegree of either 
$S^{\beta}/M$ or $S_A[\del_A^{-1}]/T^{\beta}$. 
Thus we obtain the result from long exact sequences in 
Euler--Koszul homology and Proposition~\ref{prop-HH qis 0}.
\end{proof}

\begin{definition}\label{def-E--K char}
The {\it $t^\text{th}$ partial Euler--Koszul characteristic}
of a weakly toric module $N$ is the (nonnegative) integer
\[
\chi_t(N,\beta) = \sum_{q=t}^d\ (-1)^{q-t}\cdot\rank\HH_q(N,\beta).
\]
\end{definition}

The main result of this article, Theorem~\ref{thm-main}, 
states that the partial Euler--Koszul characteristics of ranking toric modules 
are determined by the combinatorics of the ranking lattices $\bbE^{\beta}$ 
of $M$ at $\beta$. 
Lemma~\ref{lemma-chi P} describes $j(\beta)$ as the 
$2^\text{nd}$ partial Euler--Koszul characteristic of $P^{\beta}$, 
so our results regarding the combinatorics of rank jumps 
are a consequence of Theorem~\ref{thm-main}. 

\begin{lemma}\label{lemma-chi P}
The $0^\text{th}$ partial Euler--Koszul characteristic of every 
nontrivial ranking toric module for a pair $(M,\beta)$ is 0. 
In particular, $j(\beta) = \chi_2(P^{\beta},\beta)$.
\end{lemma}
\begin{proof}
This follows from \cite[Theorem~4.2]{okuyama} and \eqref{5}.
\end{proof}

\subsection{Simple ranking toric modules}
\label{subsec:simple rtm}

The polyhedral structure of the degree sets of 
ranking toric modules plays an important role in 
our computation of their partial Euler--Koszul characteristics.
We will use that each face $F \preceq A$ determines an 
$S_F$-module that is the quotient of a ranking toric module. 
Modules of this type are called 
\emph{simple ranking toric modules}.   

\begin{definition}\label{def-rtoric simple} \ \
\begin{enumerate}
\item For a subset $J\subseteq\cJ(\beta)$, 
	the ranking toric module $P_J^{\beta}$ is \emph{simple} 
	if there is a unique $F\in \FF(\beta)$ such that all pairs 
	in $J$ are of the form $(F,b)$. 
\item For each $F\in\FF(\beta)$ and $J\subseteq\cJ(\beta)$, 
	denote by $P_{F,J}^{\beta}$ the simple ranking toric module 
	determined by the set $\{(G,b)\in J \mid F = G\}$. 
	The degree set of this module is denoted by $\PP_{F,J}^{\beta}$. 
\item Call the parameter $\beta$ \emph{simple} for $M$ if $P^{\beta}$ is a 
	simple ranking toric module, or equivalently, 
    if there is an $F\in\FF(\beta)$ such that 
    $P^{\beta} = P_F^{\beta}$ (see Notation~\ref{not-omitJ}).
\end{enumerate}
\end{definition}

We show in Theorem~\ref{thm-rank simple R-toric} that for $F\preceq A$, 
each simple ranking toric module $P_{F,J}^{\beta}$ 
is a maximal Cohen--Macaulay toric $S_F$-module. 
Thus the results of Section~\ref{sec:EKstr} can be applied to compute 
the rank of their Euler--Koszul homology modules. 

Notice that by setting 
\begin{eqnarray*}
B_{F,J}^{\beta} = \{b\in B_F^{\beta}\mid (F,b)\in J\} 
\quad \text{and} \quad 
\bbE_{F,J}^{\beta} = \bigsqcup_{b\in B_{F,J}^{\beta}} (b+\Z F), 
\end{eqnarray*}
it follows from Definition~\ref{def-rtoric simple} that 
$\PP_{F,J}^{\beta} = \cC_A(\beta)\cap\bbE_{F,J}^{\beta}$. 
In particular, when $J = \cJ(\beta)$, 
$\PP_F^{\beta} =  \cC_A(\beta)\cap\bbE_{F}^{\beta}$.

\begin{proposition}\label{prop-simple r-t}
    For $F\in\FF(\beta)$ and $J\subseteq\cJ(\beta)$, 
    the simple ranking toric module $P_{F,J}^{\beta}$ 
    of Definition~\ref{def-rtoric simple} admits 
    an $S_F$-module structure that is compatible with 
    its $S_A$-module structure.
\end{proposition}
\begin{proof}     
By construction, $\PP_{F,J}^{\beta}$ is closed under addition with elements of $\N F$.
\end{proof}

\begin{definition}\label{def-triangle}
We define a partial order $\trianglelefteq$ on 
$J \subseteq \cJ(\beta)$ by 
$(F,b)\trianglelefteq (F',b')$ if and only if 
$b+\Z F \subseteq b'+\Z F'$ for pairs $(F,b),(F',b')\in J$.
We let $\max(J)$ denote the subset of $J$ consisting of maximal elements 
with respect to $\trianglelefteq$. 
\end{definition}
For $J\subseteq\cJ(\beta)$, 
$\max(J)$ is the smallest subset of $J$ that determines 
a direct sum of simple ranking toric modules into which 
$P_J^{\beta}$ embeds. 

The calculation of the partial Euler--Koszul characteristics of a 
ranking toric module $P_J^{\beta}$ will be achieved by 
homologically replacing it by an acyclic complex $I_J^\bullet$ 
composed of simple ranking toric modules. 
We then examine the spectral sequences 
determined by the double complex $\KK_\bullet(I_J^\bullet,\beta)$ 
to obtain a formula for the partial Euler--Koszul characteristics of $P_J^{\beta}$. 

\subsection{A reduction useful for computations}

We now define an equivalence relation on the union of the various $\Z F$-orbit 
representatives of \eqref{eqn-partition bbE}.
We show in Proposition~\ref{prop-P decomp} that for 
$J\subseteq\cJ(\beta)$, 
the ranking toric module $P_J^{\beta}$ splits 
as direct sum over the equivalence classes of this relation. 
Thus by additivity of rank, \eqref{5} can be expressed as a sum involving 
simpler ranking toric modules. 

\begin{definition}\label{def-tilde}\ \ 
\begin{enumerate}
\item Let 
	$B^{\beta} = 
		\bigcup_{F\in\FF(\beta)} B_F^{\beta}$
	be the collection of all $\Z F$-orbit representatives 
	from \eqref{eqn-partition bbE}.
\item Let $\bumpeq$ be the equivalence relation on the elements 
	of $B^{\beta}$ that is generated by the relations 
	$b \bumpeq b'$ if there exist $(F,b),(F',b')\in\cJ(\beta)$ 
	such that $(b+\Z F)\cap(b'+\Z F')\neq\nothing.$
\item Let $\widehat{B}^{\beta}$ denote the set of 
	equivalence classes of $\bumpeq$.
\item For $\widehat{b}\in\widehat{B}^{\beta}$ and 
	$J\subseteq\cJ(\beta)$, let 
	\[
	J(\widehat{b}) = \{(F,b')\in J \mid b'\in\widehat{b}\}. 
	\]
	Hence for $J\subseteq\cJ(\beta)$, there is a partition of 
	$\PP_J^{\beta}$ over $\widehat{B}^{\beta}$: 
\ 	$\PP_J^{\beta} = 
       \bigsqcup_{\widehat{b}\in\widehat{B}^{\beta}} \PP_{J(\widehat{b})}^{\beta}$.
\end{enumerate}
\end{definition}

\begin{proposition}\label{prop-P decomp}
For $J\subseteq\cJ(\beta)$, there is a decomposition 
$P_J^{\beta} = 
\displaystyle\bigoplus_{\widehat{b}\in\widehat{B}^{\beta}}P_{J(\widehat{b})}^{\beta}$.
\end{proposition}
\begin{proof}
For distinct $\widehat{b},\widehat{b'}\in B^{\beta}$, 
the sets $\PP_{J(\widehat{b})}^{\beta}$ 
and $\PP_{J(\widehat{b'})}^{\beta}$ are disjoint by definition of $\bumpeq$. 
Thus there is a decomposition 
$P_J^{\beta} = T^{\beta}/S_J^{\beta} = 
\bigoplus_{\widehat{b}\in B^{\beta}} T^{\beta}/S_{J(\widehat{b})}^{\beta}$. 
\end{proof}

\begin{example}
As a special case of Proposition~\ref{prop-P decomp}, 
the simple ranking toric module $P_F^{\beta}$ can be expressed as the direct sum 
$\bigoplus_{b\in B_F^{\beta}} P_{(F,b)}^{\beta}$ (see \eqref{eqn-partition bbE}). 
\end{example}

\begin{definition}
For $J = \cJ(\beta)$, let the {\it rank jump from $\widehat{b}$ of $M$ at $\beta$} be
    \begin{eqnarray*}
    j_{\widehat{b}}(\beta) = \rank\HH_1(P_{\cJ(\widehat{b})}^{\beta},\beta)
    -\rank\HH_0(P_{\cJ(\widehat{b})}^{\beta},\beta). 
    \end{eqnarray*}
\end{definition}

\begin{corollary}\label{cor-j decomp}
The rank jump $j(\beta)$ can be expressed as the sum
$j(\beta) =  \sum_{\widehat{b}\in\widehat{B}^{\beta}} j_{\widehat{b}}(\beta)$.
\end{corollary}
\begin{proof}
This follows from \eqref{5}, Proposition~\ref{prop-P decomp},  
and the additivity of $j(\beta)$.
\end{proof}

As stated in Corollary~\ref{cor-j decomp}, computing $j(\beta)$ 
is reduced to finding $j_{\widehat{b}}(\beta)$ for each 
$\widehat{b}\in\widehat{B}^{\beta}$. 
When working with examples, it is typically useful to consider 
each $j_{\widehat{b}}(\beta)$ individually. 
In contrast, as we continue with the theory, it is more efficient for 
our notation to study $j(\beta)$ directly. 
In Sections~\ref{subsec:simple case} and~\ref{subsec:non-simple case}, 
replacing $\FF(\beta)$, $P^{\beta}$, and $j(\beta)$ 
by their corresponding $\widehat{b}$ counterparts calculates 
$j_{\widehat{b}}(\beta)$.

\section{Partial Euler--Koszul characteristics of ranking toric modules}
\label{sec:comb of rank}

We retain the notation of Section~\ref{sec:r-t mods}. 
This section contains our main result, Theorem~\ref{thm-main}, 
which states that for any subset $J\subseteq\cJ(\beta)$, 
the partial Euler--Koszul characteristics of the 
ranking toric module $P_J^{\beta}$ are determined by 
the combinatorics of $\bbE_J^{\beta}$ 
(refer to Definitions~\ref{def-E--K char},~\ref{def-TS and r-t}, and~\ref{def-EP}). 
As a special case of this result, 
we will have computed $j(\beta) = \chi_2(P^{\beta},\beta)$ 
in terms of the ranking lattices $\bbE^{\beta}$, 
resulting in a proof of Theorem~\ref{thm-compute jump}. 

We begin by examining the partial Euler--Koszul characteristics 
of simple ranking toric modules 
$\PP_F^{\beta}$ from Definition~\ref{def-rtoric simple}. 
We will compute the partial Euler--Koszul characteristics of 
a ranking toric module $P_J^{\beta}$ by homologically approximating it by 
a cellular resolution (see Definition~\ref{def-cellular definition}) 
built from simple ranking toric modules. 

\subsection{The simple case}
\label{subsec:simple case}

The next theorem shows that simple ranking toric modules 
are maximal Cohen--Macaulay toric face modules, 
which will be useful in the general case.
This allows us to compute the rank jump $j(\beta)$ of $M$ at $\beta$ 
when $\beta$ is simple for $M$, as in \cite[Theorem~2.5]{okuyama}.

\begin{theorem}\label{thm-rank simple R-toric}
Fix $\beta\in\C^d$, $F\in\FF(\beta)$, 
and $J\subseteq\cJ(\beta)$. 
Then the simple ranking toric module 
$P_{F,J}^{\beta}$ is a maximal Cohen--Macaulay toric $S_F$-module.  
Further, there is a decomposition
\begin{eqnarray}\label{10}
\HH_\bullet(P_{F,J}^{\beta},\beta) =
    \C[x_{F^c}] \otimes_\C \HH^F_0(P_{F,J}^{\beta},\beta) 
    \otimes_\Z \left(\textstyle\bigwedge^\bullet \Z F^\bot\right), 
\end{eqnarray}
and for all $q\geq0$,
\[
\rank \HH_q(P_{F,J}^{\beta},\beta)
	=|B_{F,J}^{\beta}|\cdot\binom{\codim(F)}{q}\cdot \vol(F).
\]
\end{theorem}
\begin{proof}
Fix a $\Z F$-orbit representative $b\in B^{\beta}$,
chosen so that $\PP_{F,b}^{\beta}\subseteq b+\wt{\N F}$.
This implies that 
\begin{eqnarray}\label{8}
0 \rightarrow P_{F,b}^{\beta} \rightarrow
    \wt{S}_F(b) \rightarrow
    \dfrac{\wt{S}_F(b)}{P_{F,b}^{\beta}} \rightarrow 0
\end{eqnarray}
is a short exact sequence of toric modules. Since
\[
\deg\left( \dfrac{\wt{S}_F(b)}{P_{F,b}^{\beta}} \right)
    = (b+\wt{\N F})\setminus\PP_{F,b}^{\beta},
\]
the definition of $\cC_A(\beta)$ ensures that
$\beta \notin \qdeg\left( \dfrac{\wt{S}_F(b)}{P_{F,b}^{\beta}} \right)$.
Proposition~\ref{prop-HH qis 0} and 
\eqref{shift HH} imply that \eqref{8} induces the isomorphism
$\HH_\bullet(P_{F,b}^{\beta},\beta) \cong \HH_\bullet^F(\wt{S}_F,\beta-b)(b)$. 
As $\beta-b\in\C F$, Proposition~\ref{prop-str CM EK complex} 
gives the decomposition \eqref{10}, in light of Proposition~\ref{prop-P decomp}.
By \cite[Lemma~3.3]{uli monodromy}, 
$\rank\HH_0^F(\wt{S}_{F},\beta-b)(b)=\vol(F)$.
Now the additivity of rank and \eqref{3} combine to complete the claim.
\end{proof}

\begin{corollary}\label{cor-simple}
If $\beta\in\EE_A(M)$ is simple for $M$,
then the rank jump of $M$ at $\beta$ is 
\[
j_{A}(\beta) = |B^{\beta}| \cdot [\codim(F)-1] \cdot \vol(F).
\]
\end{corollary}
\begin{proof}
Since $\beta$ is simple for $M$, 
$P^{\beta} = P_F^{\beta}$ for some $F\preceq A$. 
Hence apply Theorem~\ref{thm-rank simple R-toric} 
to Corollary~\ref{cor-j decomp}, 
noting that $B^{\beta} = B_F^{\beta}$. 
\end{proof}

\begin{example}
(Continuation of Examples~\ref{example-hidden} and~\ref{example-hidden2})
\label{example-hidden3}
With $b = \left[\begin{smallmatrix} 1\\1\\0 \end{smallmatrix}\right]$, 
we have the set 
$\widehat{B}^{\beta} = \left\{\widehat{\beta},\ \widehat{b} \right\}$. 
Both $P_{\cJ(\widehat{\beta})}^{\beta}$ and $P_{\cJ(\widehat{b})}^{\beta}$ 
are simple ranking toric modules. 
By Corollary~\ref{cor-simple}, 
\begin{eqnarray*}
j_{\widehat{\beta}}(\beta)  =  1\cdot [2-1]\cdot 1 \ = \ 1 
\quad \text{and} \quad
j_{\widehat{b}}(\beta)  =  1\cdot [1-1]\cdot 1 \ = \ 0.
\end{eqnarray*}
It now follows from Proposition~\ref{prop-P decomp} that $j(\beta) = 1$. 
A similar calculation shows that $j(\beta')=1$ 
for any $\beta'\in\EE_A(S_A)$.
\end{example}

\begin{example}\label{example-4diml hidden}
Let 
\[
A=\left[ \begin{array}{ccccccccccc} 
	2 & 3 & 0 & 0 & 1 & 0 & 1 & 0 & 1  & 0 & 1 \\
	0 & 1 & 2 & 0 & 0 & 1 & 1 & 0 & 0 & 1 & 1 \\
	0 & 0 & 0 & 1 & 1 & 1 & 1 & 0 & 0 & 0 & 0 \\
	0 & 0 & 0 & 0 & 0 & 0 & 0 & 1 & 1 & 1 & 1 	
	\end{array}\right]
\]
with $\vol(A) = 24$,
and consider the faces $F=[ a_1\ a_2\ a_3 ]$ and $G = [ a_3 ]$. 
Note that the semigroup $\N A$ of 
Example~\ref{example-hidden} embeds into the $\N A$ here.
With $b = \left[\begin{smallmatrix} 1\\1\\0\\0 \end{smallmatrix}\right]$, 
the exceptional arrangement of $S_A$ is $\EE_A(S_A) = b + \C F$, where  
\[
- \qdeg\left( \ext_{R}^{i}(S_A,R)( - \varepsilon_A) \right) = 
    \begin{cases}
    b+\C F & \text{ if $i = 8$,}\\
    b+\C G & \text{ if $i = 9$,}\\
    \nothing & \text{if $i > 9$.} 
\end{cases}
\]
Thus the ranking slab stratification of $\EE_A(S_A)$ is strictly finer than 
its arrangement stratification. 
Further, this finer stratification coincides with the rank stratification 
of $\EE_A(S_A)$ inside the cone $\R_{\geq0}A$. 
For $\beta\in b+\C G$, $|\widehat{B}^{\beta}| = 2$, 
while $|\widehat{B}^{\beta}| = 1$ for $\beta\in \EE_A(S_A)\setminus[b+\C G]$.  
Calculations similar to those of Example~\ref{example-hidden3} show that 
\[
j(\beta) = 
    \begin{cases}
    3 & \text{if $\beta\in b + \C G$,} \\
    1 & \text{if $\beta\in \EE_A(S_A)\setminus[b+\C G]$.}
    \end{cases}
\]
In particular, the rank of the $A$-hypergeometric system 
$\HH_0(S_A,\beta) = M_A(\beta)$ is not constant on the slab 
$[b+\C F]\subseteq \EE_A(S_A)$ (see Definition~\ref{def:slab}).
\end{example}

\begin{example}\label{example-2 lines C4}
Let $M= S_A$ for
\[
A = \left[\begin{array}{cccccccc}
    2&1&1&1&1&1&1&1\\
    0&1&1&1&1&0&0&1\\
    0&0&1&3&4&0&1&2\\
    0&0&0&0&0&1&1&1
    \end{array}\right]
\]
and consider the saturated faces $F_1=[a_1]$ and $F_2=[a_8]$.
Here $\vol(A) = 20$, $\vol(F_i) = 1$, and $\codim(F_i) = 3$.
Computations in Macaulay 2 \cite{M2} with \eqref{13} reveal that
\[
\EE_A(S_A) = [\beta'+\C F_1]\cup[\beta'+\C F_2],
\]
where
$\beta' = \left[
       \begin{smallmatrix} 1 \\ 1 \\ 2 \\ 0\end{smallmatrix}
       \right]$.
With $b =  \left[
       \begin{smallmatrix} 1 \\ 0 \\ 0 \\ 0\end{smallmatrix}
       \right]$ and 
$\beta\in\wt{\N A}\cap\EE_A(S_A)$, 
\[
\cR_A(S_A,\beta)=
    \begin{cases}
    \EE_A(S_A) & \text{if $\beta=\beta'$}\\
    \beta+\C F_1 &
        \text{if $\beta\in[\beta'+\C F_1]\setminus\beta'$,} \\
    \beta+\C F_2 &
        \text{if $\beta\in[\beta'+\C F_2]\setminus\beta'$,}
    \end{cases}
\]
\[
\PP^{\beta}=
    \begin{cases}
    [\beta+b+\N F_1]\cup [\beta+\N F_1]\cup 
    [\beta+\N F_2] & \text{if $\beta=\beta'$}\\
    [\beta+b+\N F_1]\cup \beta+\N F_1 & 
    		\text{if $\beta\in\wt{\N A}\cap[\beta'+\C F_1]\setminus\beta'$,} \\
    \beta+\N F_2 & \text{if $\beta\in\wt{\N A}\cap[\beta'+\C F_2]\setminus\beta'$,}
    \end{cases}
\]
and 
\[
\widehat{B}^{\beta}=
    \begin{cases}
    \{    \widehat{\beta+b}, \widehat{\beta} \} & \text{if $\beta=\beta'$}\\
    \{ \widehat{\beta+b}, \widehat{\beta} \} & 
    		\text{if $\beta\in\wt{\N A}\cap[\beta'+\C F_1]\setminus\beta'$,} \\
  \{ \widehat{\beta} \}  & \text{if $\beta\in\wt{\N A}\cap[\beta'+\C F_2]\setminus\beta'$.}
    \end{cases}
\]

By Corollary~\ref{cor-simple}, 
for $\beta\in \wt{\N A}\cap[\beta'+\CC F_2]\setminus\beta'$,
\[
j(\beta)\ =\
        [\codim(F_2)-1]\cdot\vol(F_2)\ =\
        [3-1]\cdot 1\ =\ 2,
\]
while for $\beta\in \wt{\N A}\cap[\beta'+\CC F_1]\setminus\beta'$, 
$|B_{F_1}^{\beta}| = 2$, and 
\[
j(\beta)\ =\
        2\cdot[\codim(F_1)-1]\cdot\vol(F_1)\ =\
        2 \cdot [3-1]\cdot 1\ =\ 4. 
\]
To compute the rank jump of $S_A$ at $\beta'$, 
we must move to the general case. 
We will see in Example~\ref{example-2 lines C4 again} 
that $j(\beta')\ =\ 4$, which arises as the sum of the generic rank jumps 
along irreducible components of $\cR_A(S_A,\beta')$ that is then 
corrected by error terms that arise from a spectral sequence calculation. 
\end{example}

\subsection{The general case}
\label{subsec:non-simple case}

We are now prepared to compute the 
partial Euler--Koszul characteristics of ranking toric modules. 
The proof of our main theorem, Theorem~\ref{thm-main}, will be given at 
the end of this section, after a sequence of lemmas. 
The definitions of $\cJ(\beta)$, $\bbE_J^{\beta}$, and $P_J^{\beta}$ can be 
found in Definitions~\ref{def-EP}~and~\ref{def-TS and r-t}, respectively. 

\begin{theorem}\label{thm-main}
For $J\subseteq\cJ(\beta)$, the partial Euler--Koszul characteristics 
of the ranking toric module $P_J^{\beta}$ 
are determined by the combinatorics of $\bbE_J^{\beta}$.
\end{theorem}

We compute the partial Euler--Koszul characteristics 
of the ranking toric module $P_J^{\beta}$ 
will be achieved by homologically approximating $P_J^{\beta}$ 
by simple ranking toric modules; 
note that the ranks of the Euler--Koszul homology modules 
of simple ranking toric modules 
have been computed in Theorem~\ref{thm-rank simple R-toric}. 

\begin{definition}\label{def-cellular definition}
Let $\Delta$ be an oriented cell complex (e.g. CW, simplicial, polyhedral). 
Then $\Delta$ has the cochain complex 
\[
C^\bullet_\Delta :  \quad
0 \rightarrow \bigoplus_{\text{vertices } v \in\Delta } \Z_v \rightarrow 
\bigoplus_{\text{edges } e \in\Delta } \Z_e \rightarrow 
\cdots \rightarrow 
\bigoplus_{ \text{$i$-faces } \sigma \in \Delta} \Z_\sigma 
\rightarrow \cdots \rightarrow 0, 
\]
where $\Z_\sigma \rightarrow \Z_\tau$ 
is multiplication by some integer coeff$(\sigma,\tau)$. 
Let $\fC_\Delta$ be the category with the nonempty faces of $\Delta$ as objects 
and morphisms 
\[
\mor_{\fC_\Delta}(\sigma,\tau) = 
\begin{cases}
\{\sigma \subseteq \tau\} & \text{if $\sigma \subseteq \tau$,}\\
\nothing & \text{otherwise.}
\end{cases}
\]
Fix an abelian category $\fA$ and 
suppose there is a covariant functor 
$\Phi:\fC_\Delta \rightarrow \fA$. 
Let $P_\sigma := \Phi(\sigma)$ for each $\sigma\in\Delta$. 
A sequence of morphisms in $\fA$ 
\[
I^\bullet: \quad 
0 \rightarrow \bigoplus_{\text{vertices } v \in\Delta } P_v  \rightarrow 
\bigoplus_{\text{edges } e \in\Delta } P_e \rightarrow 
\cdots \rightarrow 
\bigoplus_{ \text{$i$-faces } \sigma \in \Delta} P_\sigma 
\rightarrow \cdots \rightarrow 0
\]
is \emph{cellular} and \emph{supported on}~$\Delta$ if 
the $P_\sigma \rightarrow P_\tau$ component 
of $I^\bullet$ is 
coeff$(\sigma,\tau)$ $\Phi(\sigma\subseteq\tau)$. 
Since $\fA$ is abelian, a cellular sequence is necessarily a complex. 
\end{definition}

In a manner analogous to Definition~\ref{def-cellular definition}, 
a cellular complex supported on $\Delta$ can also be obtained 
from the chain complex $C^\Delta_\bullet$ of $\Delta$ 
and a contravariant functor $\Phi: \fC_\Delta \rightarrow \fA$. 
Further, we could replace $C^\Delta_\bullet$ in this construction 
with the \emph{reduced} chain or cochain complexes of $\Delta$. 
We say that a complex in $\fA$ is \emph{cellular} if 
it can be constructed from the underlying topological data of a cell complex 
and a functor $\Phi:\fC_\Delta \rightarrow \fA$. 

When $\Delta$ is a simplicial or polyhedral cell complex, 
coeff$(\sigma,\tau)$ of Definition~\ref{def-cellular definition} is simply 
$1$ if the orientation of $\tau$ induces the orientation of $\sigma$ 
and $-1$ if it does not.

The generality in which we define cellular complexes 
is alluded to in the introduction of \cite{jm cellular} 
and appears as \cite[Definition~3.2]{miller-survey}. 
An introduction to these complexes, in the polyhedral case,
can be found in \cite[Chapter~4]{MS05}.

Recall from Definition~\ref{def-triangle} that $\max(J)$ was defined so that 
it yields the smallest set of faces of $A$ that determines 
a direct sum of simple ranking toric modules into which 
$P_J^{\beta}$ embeds. 

\begin{notation}\label{not-cpx I}
We wish to take intersections of faces in the set $\max(J)$. 
In order to keep track of which faces were involved in each intersection, set 
\begin{eqnarray*}
\cG^0_J &=& \{ F\in\FF(\beta) \mid \exists (F,b)\in \max(J) \}, \\ 
\cG^p_J &=& \{ s \subseteq \cG^0_J \mid |s| = p+1\}, \text{ and} \\
F_s &=& \bigcap_{G\in s} G \ \ \text{for $s\in\cG^p_J$}. 
\end{eqnarray*} 
With $r +1 = |\Delta^0_J|$, 
let $\Delta = \Delta_J^{\beta}$ be the standard $r$-simplex  
with vertices corresponding to the elements of $\Delta^0_J$. 
To the $p$-face of $\Delta$ spanned by the vertices 
corresponding to the elements in $s\in\Delta^p_J$, 
assign the ranking toric module $P_{F_s,J}^{\beta}$. 
Choosing the natural maps 
$P_{F_s,J}^{\beta} \rightarrow P_{F_t,J}^{\beta}$ 
for $s\subseteq t$ 
induces a cellular complex supported on $\Delta$, 
\begin{eqnarray}\label{9}
    I_J^\bullet:\quad 
    I_J^0 \rightarrow
    I_J^1 \rightarrow \cdots \rightarrow I_J^r \rightarrow 0
\end{eqnarray}
with
\begin{eqnarray*}
I_J^p = \bigoplus_{s\in\cG^p_J} P_{F_s,J}^{\beta}.
\end{eqnarray*}
\end{notation}

\begin{lemma}\label{lemma-I acyclic}
The cohomology of the cellular complex $I_J^\bullet$ of \eqref{9} 
is concentrated in cohomological degree 0 
and is isomorphic to $P_J^{\beta}$. 
\end{lemma}
\begin{proof}
Given $\alpha \in \PP_J^{\beta} = \deg(P_J^{\beta})$, 
let $F_{i_1}, \dots , F_{i_k}$ 
be the faces $F \in \cG^0_J$ such that 
$\alpha \in \PP_{F,J}^{\beta}$. 
The degree $\alpha$ part of $I_J^\bullet$ computes the 
cohomology of the $(k - 1)$-subsimplex of $\Delta$ 
given by the vertices with labels corresponding to $F_{i_1}, \dots , F_{i_k}$; 
in particular, it is acyclic with $0$-cohomology 
$\C \cong (P_J^{\beta})_\alpha$. 
\end{proof}

By construction, $P_J^{\beta}$ is a 
$\Z^d$-graded monomial module over the 
saturated semigroup ring $\wt{S}_A$, and 
it can be translated by some $\alpha\in\Z^d$ so that 
$\deg(P_J^{\beta}(\alpha)) = 
\alpha + \PP_J^{\beta} \subseteq \wt{\N A} = \deg( \wt{S}_A )$. 
After translation by $\alpha$, \eqref{9} is similar to an 
\emph{irreducible resolution}, 
as defined in \cite[Definition~2.1]{cm quotients}.
We continue to view $P_J^{\beta}$ 
as an $S_A$-module, so we use 
maximal Cohen--Macaulay toric face modules 
instead of irreducible quotients of $\wt{S}_A$. 

Consider the $\Z^d$-graded double complex $E_0^{\bullet,\bullet}$ with
$E_0^{p,-q}:=\KK_q(I_J^p,\beta).$
Let ${}_h\psi_0$ and ${}_v\psi_0$ denote the horizontal and
vertical differentials of $E_0^{\bullet,\bullet}$, respectively.
By the exactness of \eqref{9}, taking homology of
$E_0^{\bullet,\bullet}$ with respect to ${}_h\psi_0$ yields
\begin{align}
{}_hE_1^{p,-q} &=
    \begin{cases} \KK_q(P_J^{\beta},\beta) &
        \text{if $p=0\text{ and }0\leq q \leq d,$}\\
    0 & \text{otherwise.}
    \end{cases} \nonumber
\intertext{Hence}
\label{36}
{}_hE_\infty^{p,-q}={}_hE_2^{p,-q} &=
    \begin{cases}
    \HH_q(P_J^{\beta},\beta) &
        \text{if $p=0\text{ and }0\leq q \leq d,$}\\
    0 & \text{otherwise.}
    \end{cases}
\intertext{On the other hand, the first page of the vertical spectral sequence 
given by $E_0^{\bullet,\bullet}$ consists of Euler--Koszul homologies 
of simple ranking toric modules:  }
\label{23}
{}_vE_1^{p,-q} &= \HH_q(I_J^p,\beta)
    = \bigoplus_{s\in\cG^p_J}
    \HH_q(P_{F_s,J}^{\beta},\beta).
\end{align}
We now apply the decomposition of these homologies given in 
Theorem~\ref{thm-rank simple R-toric} to obtain a new description 
of the Euler--Koszul homology of the ranking toric module $P_J^{\beta}$. 

\begin{lemma}\label{lemma-right ss}
The vertical spectral sequence obtained from the double complex 
\begin{eqnarray}\label{32}
\phantom{xxxx}
'E_0^{p,-q} & = & \bigoplus_{s\in\cG^p_J} \bigoplus_{i+j=q}
    \ \C[x_{F_s^c}] \otimes_\C
    \KK^{F_s}_i(P_{F_s,J}^{\beta},\beta) \otimes_\Z
    \left( \textstyle\bigwedge^j \Z F_s^\bot \right)
\end{eqnarray}
(with differentials as in Lemma~\ref{lemma-EK morph}) 
has abutment 
\[
'E_\infty^{p-q} \cong 
\begin{cases} 
\HH_q(P_J^{\beta},\beta) & \text{if $p = 0$,}\\
0 & \text{otherwise.}
\end{cases}
\]
\end{lemma}
\begin{proof}
By Theorem~\ref{thm-rank simple R-toric} and Lemma~\ref{lemma-EK morph}, 
the vertical differentials ${}_v\psi_0$ of $E_0^{p,-q}$ are compatible with
the quasi-isomorphism 
\begin{eqnarray}\label{eqn-ss qis}
E_0^{p,\bullet} \simeq_{qis}
\bigoplus_{s\in\cG^p_J}
    \C[x_{F_s^c}] \otimes_\C
    \KK^{F_s}_\bullet(P_{F_s,J}^{\beta},\beta) \otimes_\Z
    \left( \textstyle\bigwedge^\bullet \Z F_s^\bot \right).
\end{eqnarray}
Since ${}_hE_\infty^{\bullet,\bullet}$ and ${}_vE_\infty^{\bullet,\bullet}$ 
converge to the same abutment, the result follows from \eqref{36}. 
\end{proof}

Note that the first page of the spectral sequence in Lemma~\ref{lemma-right ss} is
\begin{eqnarray}\label{eq:nice 1st page}
'E_1^{p,-q} & = & \bigoplus_{s\in\cG^p_J}
    \C[x_{F_s^c}] \otimes_\C
    \HH^{F_s}_0(P_{F_s,J}^{\beta},\beta) \otimes_\Z
    \left( \textstyle\bigwedge^q \Z F_s^\bot \right).
\end{eqnarray}

For $s\in\cG^p_J$, let $\kappa_s$ denote the differential of
$\KK_\bullet^{F_s}(P_{F_s,J}^{\beta},\beta)$, and let
${}_v\delta, {}_h\delta$ respectively denote the vertical and 
horizontal differentials of ${'E}_0^{\bullet,\bullet}$. 
If $i+j=q$ with $i,j\geq0$, then by 
Remark~\ref{remark-diffl in CM EK complex}, 
the element
\[
f\otimes a\otimes b \in
    \C[x_{F_s^c}] \otimes_\C
    \KK_i^{F_s}(P_{F_s,J}^{\beta},\beta) \otimes_\C
    \left( \textstyle\bigwedge^q \Z F_s^\bot \right)
    \subseteq \ {'E}_0^{p,-q}
\]
has vertical differential
\begin{eqnarray}\label{31}
{}_v\delta (f\otimes a\otimes b) = f\otimes \kappa_s(a)\otimes b.
\end{eqnarray}
We will use that \eqref{31} is an element of
\[
\C[x_{F_s^c}]\otimes
\KK_{i-1}^{F_s} (P_{F_s,J}^{\beta},\beta)
\otimes_\C
\left( \textstyle\bigwedge^j \Z F_s^\bot \right)
\subseteq {'E}_0^{p,-q+1}
\]
to show that ${'E}_\bullet^{\bullet,\bullet}$ degenerates quickly. 
This is the main technical result of this article. 

\begin{lemma}\label{lemma-ss terminates}
The spectral sequence $'E_\bullet^{\bullet,\bullet}$ of
Lemma~\ref{lemma-right ss} degenerates at the second page. 
\end{lemma}
\begin{proof}
For $\xi\in {'E}_0^{i,j}$, let $\ol{\xi}$ denote the image of $\xi$ 
in ${'E}_2^{i,j}$, if it exists. 
Let $\delta_r$ denote the differential of ${'E}_r^{\bullet,\bullet}$, 
so $\delta_0 = {}_v\delta$. 

To see that $\delta_2=0$, consider an element
$\alpha\in {'E}_0^{p,-q}$ with $\oalpha\in {'E}_2^{p,-q}$.
Then there is an element $\eta\in {'E}_0^{p+1,-q-1}$
such that ${}_v\delta(\eta)={}_h\delta(\alpha)$, which is used to define
\begin{eqnarray}\label{40}
\delta_2(\oalpha)=\ol{{}_h\delta(\eta)}.
\end{eqnarray}
(Recall that \eqref{40} is independent of the choice of $\eta$.)
We write $\alpha=\sum \alpha^s_{ij}$ as an element of \eqref{32}.
Note that 
\begin{align}
{}_v\delta(\alpha^s_{ij}) &\in \C[x_{F_s^c}] \otimes_\C
\KK_{i-1}^{F_s}(P_{F_s,J}^{\beta},\beta) 
\otimes_\Z \left( \textstyle\bigwedge^j \Z F_s^\bot \right), 
\nonumber 
\intertext{so $\alpha^s_{ij}$ is in the kernel of ${}_v\delta$ for all $s,i,j$. 
By \eqref{eq:nice 1st page}, $\alpha^{s}_{ij}$ is in the image of ${}_v\delta$ 
whenever $i>0$.
Hence without changing $\oalpha$, we may assume that for all $s\in\cG^p_J$, 
$\alpha^{s}_{ij}=0$ when $i>0$, so that}
\alpha &\in \bigoplus_{s\in\cG^p_J}
\C[x_{F_s^c}] \otimes_\C \KK_0^{F_s}(P_{F_s,J}^{\beta},\beta)
\otimes_\Z \left( \textstyle\bigwedge^q \Z F_s^\bot \right).
\nonumber
\intertext{As the differential ${}_h\delta$ is induced by \eqref{9},}
{}_h\delta(\alpha) &\in \bigoplus_{s\in\cG^{p+1}_J}
\C[x_{F_s^c}] \otimes_\C
\KK_0^{F_s}(P_{F_s,J}^{\beta},\beta)
\otimes_\Z \left( \textstyle\bigwedge^q \Z F_s^\bot \right).
\nonumber
\intertext{By hypothesis on $\alpha$ and \eqref{31}, there is an element}
\label{34}
\eta &\in \bigoplus_{s\in\cG^{p+1}_J}
\C[x_{F_s^c}] \otimes_\C \KK_1^{F_s}(P_{F_s,J}^{\beta},\beta)
\otimes_\Z \left( \textstyle\bigwedge^q \Z F_s^\bot \right)
\intertext{such that ${}_v\delta(\eta)={}_h\delta(\alpha)$.
Set $\zeta={}_h\delta(\eta)$ and note that $\delta_2(\oalpha)=\ol{\zeta}$.
Using again that the differential ${}_h\delta$ is induced by \eqref{9}, 
applied now to \eqref{34}, we see that}
\zeta &\in \bigoplus_{s\in\cG^{p+2}_J}
    \C[x_{F_s^c}] \otimes_\C
    \KK_1^{F_s}(P_{F_s,J}^{\beta},\beta) \otimes_\Z
    \left( \textstyle\bigwedge^q \Z F_s^\bot \right).
\nonumber
\end{align}
Since ${}_v\delta(\zeta)={}_h\delta^2(\alpha)=0$, $\zeta$ is in the kernel of 
${}_v\delta=\delta_0$.
Hence \eqref{eq:nice 1st page} implies that $\delta_2(\oalpha)=\ol{\zeta}$ vanishes.
\end{proof}

\begin{lemma}\label{lemma-formula}
For $J\subseteq\cJ(\beta)$, 
the $t^\text{th}$ partial Euler--Koszul characteristic of the ranking toric module 
$P_J^{\beta}$ is given by 
\begin{align}
\chi_t(P_J^{\beta},\beta)
  = & \sum_{p-q>-t} (-1)^{p-q+t+1} \rank\HH_q(I_J^p,\beta)
 - \sum_{p-q=-t} \rank(\image \delta_1^{p,-q}).
    \label{11}
\end{align}
\end{lemma}
\begin{proof}
For $i\in\N$, let 
\begin{align*} 
r_i^k:=\sum_{p-q=k}\rank('E_i^{p,-q}).
\end{align*}
By Lemma \ref{lemma-ss terminates}, $r_2^k=r_\infty^k$.
From the abutment \eqref{36}, we see that $r_2^k=0$ for $k>0$.
Since $\sum_{k\in\Z}(-1)^kr_1^k=0$ by Lemma~\ref{lemma-chi P}, also
$\sum_{k\in\Z}(-1)^k r_2^k=0$.
Thus the $t^\text{th}$ partial Euler--Koszul characteristic of $P_J^{\beta}$ 
can be expressed as 
\begin{align*}
\chi_t(P_J^{\beta},\beta)
  = & \sum_{k=-d}^{-t} (-1)^{k+t} r_2^k \nonumber\\
  = & \sum_{k=-d}^\infty (-1)^{k+t+1} r_1^k -
         \sum_{k=-d}^{-t} (-1)^{k+t+1} r_2^k \nonumber\\
  = & \sum_{k=-t+1}^\infty (-1)^{k+t+1} r_1^k -
        \sum_{p-q=-t} \rank(\image \delta_1^{p,-q}). 
        \nonumber
\end{align*}
Now \eqref{11} follows from the definition of $r_1^k$ 
and the quasi-isomorphism \eqref{eqn-ss qis}, 
as the isomorphic first pages of the spectral sequences there are 
\eqref{23} and \eqref{eq:nice 1st page}. 
\end{proof}

We will compute the ranks of $\HH_q(I_J^p,\beta)$
and the image of $\delta_1^{p,-q}$ from 
\eqref{11} in subsequent lemmas.
The first is an immediate consequence of 
Theorem~\ref{thm-rank simple R-toric}.

\begin{lemma}\label{lemma-rank-HqIp}
If $q \geq 0$, then
\[
\rank \HH_q(I_J^p,\beta) =
    \displaystyle\sum_{s\in\cG^p_J}
    |B_{F_s}^{\beta}| \cdot
    \binom{\codim(F_s)}{q} \cdot \vol(F_s).
\]
\end{lemma}
\begin{proof}
By definition of $I_J^p$ and additivity of rank,
$\rank \HH_q(I_J^p,\beta) = 
\sum_{s\in\cG^p_J} \rank \HH_q(P_{F_s,J}^{\beta},\beta)$.
Now apply Theorem~\ref{thm-rank simple R-toric}.
\end{proof}

The rank of the image of $\delta_1^{p,-q}$ is 
determined combinatorially because the spectral sequence 
rows $'E_1^{\bullet,-q}$ are cellular complexes. 

\begin{lemma}\label{lemma-'E cell rows}
The complexes $'E_1^{\bullet,-q}$ are cellular with support 
$\Delta = \Delta^{\beta}_J$ of Notation~\ref{not-cpx I}. 
\end{lemma}
\begin{proof}
In Notation~\ref{not-cpx I}, we constructed the cellular complex 
$I_J^\bullet$ from a labeling of the simplex $\Delta = \Delta_J^{\beta}$.  
If we assign in this construction 
\[
\C[x_{F^c}] \otimes_\C \HH_q^F(P_F^{\beta},\beta) 
\otimes_\Z \left( \textstyle\bigwedge^q \Z F^\bot \right) 
\] 
in place of $P_{F,J}^{\beta}$ 
and use the induced maps, 
we obtain the cellular complex $'E_1^{\bullet,-q}$ 
with differential $\delta_1^{\bullet,-q}$, see \eqref{3}. 
The existence and compatibility of the differentials 
follows from Lemma~\ref{lemma-EK morph}. 
\end{proof}

\begin{lemma}\label{lemma-rank image}
The rank of the image of $\delta_1^{p,-q}$ is determined by 
the combinatorics of $\bbE_J^{\beta}$.
\end{lemma}
\begin{proof} 
By Lemma~\ref{lemma-'E cell rows}, 
the image of $\delta_1^{p,-q}$ is 
determined by 
the $p$-coboundaries of $\Delta$ 
and the corresponding labels of $\Delta$, 
which come from $\bbE_J^{\beta}$. 
\end{proof}

\begin{proof}[Proof of Theorem~\ref{thm-main}]
If $J$ is simple for $M$, the result follows from 
Theorem~\ref{thm-rank simple R-toric}, so suppose that 
$P_J^{\beta}$ is not simple for $M$. 
By Lemma~\ref{lemma-I acyclic}, 
$P_J^{\beta}$ is 
the $0$-cohomology of the acyclic cellular complex $I_J^\bullet$ \eqref{9}. 
Thus the abutment of the spectral sequences arising from the 
double complex $\KK_\bullet(I_J^\bullet,\beta)$ is $\HH_\bullet(P_J,\beta)$. 
By Lemma~\ref{lemma-right ss}, the vertical spectral sequence 
obtained from the double complex $'E_0^{p,-q}$ of \eqref{32} has the same abutment. 
Since this spectral sequence degenerates on the second page by 
Lemma~\ref{lemma-ss terminates}, 
Lemma~\ref{lemma-formula} yields the formula \eqref{11}, and 
By Lemmas~\ref{lemma-rank-HqIp} and~\ref{lemma-rank image}, 
the summands of \eqref{11} are dependent only 
on the combinatorics of $\bbE_J^{\beta}$. 
\end{proof}

\subsection{Computing partial Euler--Koszul characteristics}
\label{subsec:compute}

Recall formula \eqref{11}: 
\begin{align*}
\chi_t(P_J^{\beta},\beta)
  = & \sum_{p-q>-t} (-1)^{p-q+t+1} \rank\HH_q(I_J^p,\beta)
 - \sum_{p-q=-t} \rank(\image \delta_1^{p,-q}). 
\end{align*}
Lemma~\ref{lemma-rank-HqIp} computes the first summand, but 
Lemma~\ref{lemma-rank image} does not explicitly state 
the rank of the image of $\delta_1^{p,-q}$. 
A method to do this is provided by Proposition~\ref{prop:alg}. 

\begin{definition}\label{def:circuit}
For $1< j\leq |\Delta_J^p|$, 
a subset $\lambda\subseteq \{1,2,\dots,j\}$ corresponds to a subcomplex $\Delta(\lambda)$ 
of the simplex $\Delta = \Delta_J^\beta$, as described in Notation~\ref{not-cpx I}. 
If $j\in\lambda$ and there is a minimal generator of 
$H^p( \Delta(\lambda), \Delta(\lambda)\setminus\{j\}; \C)$ of the form 
$\sum_{i\in \lambda} v_i\cdot [s_i]$, where all coefficients $v_i$ are nonzero, 
then we say that $\lambda$ is a \emph{circuit} for $j$.
\end{definition}

\begin{notation}
For $1< j\leq |\Delta_J^p|$, let
\[
\Upsilon_J^p(j) = \left\{ \lambda \subseteq \{1,2,\dots,j\} 
\mid \lambda \text{ is a circuit for } j \right\}
\]
denote the set of circuits for $j$, and set
\[
\Upsilon_J^p(j,k) = \left\{ \Lambda \subseteq \Upsilon_J^p(j) \mid |\Lambda| = k\right\}. 
\]
For $s\in\Delta_J^p$, 
$\lambda\in\Upsilon_J^p(j)$, 
and $\Lambda\in\Upsilon_J^p(j,k)$, 
set 
\begin{align}
\nonumber s_\lambda &= \{s_i \mid i\in\lambda\}, \\
\nonumber F(\Lambda) &= \bigcup_{\lambda\in\Lambda} \bigcup_{i\in \lambda} F_{s_i}, \\ 
\label{eq:N}
N^p_J(s) &= \{ 
(F,b)\in J \mid \exists t \in \Delta_J^p\setminus\{s\}
\text{ with } (b+\Z F)\subseteq 
\bbE_{F_s,J}^{\beta}\cap\bbE_{F_t,J}^{\beta}\neq\nothing \}, \\ 
\nonumber N_J^p(\lambda) &= 
\left\{ (F,b) \in N_J^p(s_{j}) \ \Big\vert\ \exists i\in \lambda\setminus\{j\} 
\text{ with } (b+\Z F) \subseteq \bbE_{F_{s_i},J}^{\beta} \right\}, \\
\nonumber N_J^p(\Lambda) &= \bigcap_{\lambda\in\Lambda} N_J^p(\lambda), \\
\nonumber \nu^{p,-q}(\Lambda) &= 
	\binom{\codim \left(F(\Lambda) \right)}{q}
	\cdot \rank \HH_0( P_{N_J^p(\Lambda)}^{\beta},\beta), \ \text{and}\\
\nonumber \nu^{p,-q}(j) &= 
	\sum_{k=1}^{|\Upsilon_J^p(j)|} \sum_{\Lambda\in\Upsilon_J^p(j,k)} 
	(-1)^{|\Lambda|+1} \cdot \nu^{p,-q}(\Lambda).
\end{align}
\end{notation}

\begin{proposition}\label{prop:alg}
Let $P_{N_J^p(s)}^\beta$ be the ranking toric module in \eqref{eq:N}. 
The rank of the image of $\delta_1^{p,-q}$ from \eqref{11} is equal to 
\begin{align}\label{eq:total im}
\rank (\image & \delta_1^{p,-q}) = 
	\sum_{s\in\Delta_J^p}
		\binom{ \codim(F_s) }{ q } \cdot \rank \HH_0( P_{N_J^p(s)}^\beta, \beta) 
	- \sum_{j=2}^{|\Delta_J^p|} \nu^{p,-q}(j). 
\end{align}
Further, \eqref{eq:total im} can be computed by combining 
Theorem~\ref{thm-rank simple R-toric} and induction on the dimension 
of $P_J^\beta$.
\end{proposition}

Before providing the proof of Proposition~\ref{prop:alg}, 
we state two lemmas. 

\begin{lemma}\label{lemma:easier comp}
For $s\in \Delta_J^p$, 
\begin{align}\label{eq:1st end}
\rank (\image \delta_{1,s}^{p,-q}) = 
	\binom{ \codim(F_s) }{ q } \cdot \rank \HH_0( P_{N_J^p(s)}^\beta, \beta), 
\end{align}
where $P_{N_J^p(s)}^\beta$ is the ranking toric module given by \eqref{eq:N}. 
\end{lemma}
\begin{proof}
The rank of the image of $\delta_{1,s}^{p,-q}$ 
is $\binom{\codim(F_s)}{q} \cdot \rank (\image \delta_{1,s}^{p,0})$ 
by Proposition~\ref{prop-str/rk image}. 
View the image of $\delta_{1,s}^{p,0}$ as a quotient of $P_{F_s,J}^\beta$. 
If $\alpha$ is one of its nonzero multigraded components, 
then it also appears in the degree set of another summand of $'E_1^{p,0}$. 
The collection of such degrees is exactly $\PP_{N_J^p(s)}^\beta$.  
\end{proof}

\begin{lemma}\label{lemma:harder comp}
For $1< j\leq |\Delta_J^p|$, 
\begin{align}\label{eq:nu j}
\nu^{p,-q}(j) = 
	\rank \left[
	(\image \delta_{1,\{s_1,\dots,s_{j-1}\}}^{p,-q}) \cap 
	(\image \delta_{1,\{s_j\}}^{p,-q}) \right].
\end{align}
\end{lemma}
\begin{proof}
To see this, notice first that for $1< j \leq |\Delta_J^p|$, 
\begin{align*}
(\image \delta_{1,\{s_1,\dots,s_{j-1}\}}^{p,-q}) \cap 
	(\image  \delta_{1,\{s_j\}}^{p,-q}) 
= \sum_{\lambda\in\Upsilon_J^p(j)} 
(\image \delta_{1,s_{\lambda\setminus \{j\}}}^{p,-q}) \cap 
	(\image \delta_{1,\{s_j\}}^{p,-q}) 
\end{align*}
is generated by the images coming from circuits for $j$. 
By Proposition~\ref{prop-str/rk image}, given a fixed circuit $\lambda$ for $j$, 
the rank of 
\begin{align}\label{eq:circuit intersection q}
(\image \delta_{1,s_{\lambda\setminus \{j\}}}^{p,-q}) \cap 
	(\image \delta_{1,\{s_j\}}^{p,-q}),  
\end{align}
can be computed as the rank of 
\begin{align}\label{eq:circuit intersection 0}
(\image \delta_{1,s_{\lambda\setminus \{j+1\}}}^{p,0}) \cap 
	(\image \delta_{1,\{s_{j+1}\}}^{p,0}) 
\end{align}
times the $\Z$-rank of 
\begin{align}\label{eq:zrank}
\bigcap_{i\in\lambda} \left[ \textstyle\bigwedge^q \Z F_{s_i}^\bot \right] = 
\textstyle\bigwedge^q \Z F(\lambda)^\bot. 
\end{align} 
By the same reasoning used to obtain \eqref{eq:1st end}, the rank of 
$\HH_0( P_{N_J^p(\lambda)}^{\beta},\beta)$ equals the rank of 
\eqref{eq:circuit intersection 0}. 
The $\Z$-rank of \eqref{eq:zrank} is a binomial coefficient in 
the codimension in $\C^d$ of the span of the vectors in $F(\lambda)$, 
so the rank of \eqref{eq:circuit intersection q} is 
\[
\binom{\codim \left( F(\lambda) \right)}{q}
	\cdot \rank \HH_0( P_{N_J^p(\lambda)}^{\beta},\beta). 
\]
Further, for a collection of circuits $\Lambda\in\Upsilon_J^p(j,k)$, 
$\nu^{p,-q}(\Lambda)$ gives the rank of the intersection over $\Lambda$ 
of the images of type \eqref{eq:circuit intersection q}, 
so the inclusion-exclusion principle yields \eqref{eq:nu j}. 
\end{proof}

\begin{proof}[Proof~of~Proposition~\ref{prop:alg}]
Recall from \eqref{eq:nice 1st page} that the domain of $\delta_1^{p,-q}$ 
is the direct sum 
\begin{align*}
'E_1^{p,-q} & =  \bigoplus_{s\in\cG^p_J}
    \C[x_{F_s^c}] \otimes_\C
    \HH^{F_s}_0(P_{F_s,J}^{\beta},\beta) \otimes_\Z
    \left( \textstyle\bigwedge^q \Z F_s^\bot \right).
\end{align*}
For $S\subseteq \Delta_J^p$, let $\delta_{1,S}^{p,-q}$ denote the restriction 
of $\delta_1^{p,-q}$ to the summands in $S$. 
Order the elements of $\Delta_J^p = \{ s_1,\dots,s_{|\Delta_J^p|}\}$, so that 
\begin{align} 
\label{eq:want image}
\rank (\image \delta_1^{p,-q}) 
	 = & \sum_{s\in \Delta_J^p} \rank (\image \delta_{1,s}^{p,-q}) \\
	& - \sum_{j=2}^{|\Delta_J^p|} \rank \left[
	(\image \delta_{1,\{s_1,\dots,s_{j-1}\}}^{p,-q}) \cap 
	(\image \delta_{1,\{s_j\}}^{p,-q}) \right]. \nonumber
\end{align}
Lemmas~\ref{lemma:easier comp} and ~\ref{lemma:harder comp} respectively 
computed the summands of \eqref{eq:want image}, resulting in \eqref{eq:total im}. 
Thus it remains to show the second statement. 

If a ranking toric module has dimension 0, then it is necessarily a 
simple ranking toric module, so Theorem~\ref{thm-rank simple R-toric} 
computes the rank of its Euler--Koszul homology modules. 
Thus by induction on dimension, we can compute the summand in \eqref{eq:1st end}
corresponding to $s\in\Delta_J^p$ 
if the dimension of $P_{N_J^p(s)}$ is strictly less than the dimension of 
$P_{F_s,J}^\beta$. 

If it is the case that the dimension of $P_{N_J^p(s)}$ equals the dimension of 
$P_{F_s,J}^\beta$, 
notice first that each pair $(F,b)\in N_J^p(s)$ has $F\preceq F_s$. 
This implies that $P_{N_J^p(s)}$ is a direct sum 
(as in Proposition~\ref{prop-P decomp}) of the simple ranking toric module 
$P_{F_s,N_J^p(s)}$ and a lower-dimensional ranking toric module. 
Therefore induction together with Theorem~\ref{thm-rank simple R-toric} 
still completes the computation. 

Finally, the same argument applies to computing the rank of 
$\HH_0( P_{N_J^p(\Lambda)}^{\beta},\beta)$ for $\Lambda\in\Upsilon_J^p(j,k)$, 
since ${N_J^p(\Lambda)} \subseteq N_J^p(s_j)$. 
\end{proof}

\subsection{The combinatorics of rank jumps}
\label{subsec:rank jumps}

By Lemma~\ref{lemma-chi P}, our results on 
the partial Euler--Koszul characteristics of ranking toric modules 
reveal the combinatorial nature of rank jumps of 
the generalized $A$-hypergeometric system $\HH_0(M,\beta)$. 

\begin{proof}[Proof of Theorem~\ref{thm-compute jump}]
By \eqref{5} and Lemma~\ref{lemma-chi P}, 
$j(\beta) = \chi_2(P^{\beta},\beta)$, so the result is an 
immediate consequence of Theorem~\ref{thm-main} and 
Proposition~\ref{prop:alg}.
\end{proof}

\begin{example}\label{example-2 cmpts}
If $\beta\in\C^d$ is such that $\cG^0 = \{F_1,F_2\}$, 
the proof of Theorem~\ref{thm-main} and Section~\ref{subsec:compute} 
show that the rank jump of $M$ at $\beta$ is
\begin{align}\label{24}
j(\beta) = 
    \sum_{i=1}^2 \left(
    |B_{F_i}^{\beta}| \cdot [\codim(F_i)-1] \cdot \vol(F_i)
    \right) + |B_G^{\beta}| \cdot C^{\beta} \cdot\vol(G),
\end{align}
where 
$G = F_1\cap F_2$ and 
the constant $C^{\beta}$ is given by
\begin{eqnarray*}
C^{\beta} = \binom{\codim(G)}{2} - \codim(G) + 1 
    - \binom{\codim(F_1)}{2} - \binom{\codim(F_2)}{2} +
    \binom{\codim(\C F_1+\C F_2)}{2}. 
\end{eqnarray*}
\end{example}

\begin{example}
(Continuation of Example~\ref{example-2 lines C4})
\label{example-2 lines C4 again}
With $b' = \beta'+b = 
\left[\begin{smallmatrix}2\\1\\2\\0\end{smallmatrix}\right]$, 
the set $\widehat{B}^{\beta'} = \{\widehat{\beta'},\widehat{b'}\}$, and 
$P^{\beta'} = P_{\cJ(\widehat{\beta'})}^{\beta'} \oplus P_{\cJ(\widehat{b'})}^{\beta}$ 
by Proposition~\ref{prop-P decomp}. 
By \eqref{24}, 
\begin{eqnarray*}
j_{\widehat{\beta'}}(\beta) &=& 
 	\sum_{i=1}^2 \left(
    	|\widehat{\beta'} \cap B_{F_i}^{\beta}| \cdot [\codim(F_i)-1] \cdot \vol(F_i)
    	\right) \ +\ |\widehat{\beta'} \cap B_G^{\beta}| \cdot C^{\beta} \cdot\vol(G)\\
	&=& 2+2+1\cdot(-2)\cdot 1 \ = \ 2, 
\end{eqnarray*}
and 
$j_{\widehat{b'}}(\beta) = 2$
by Corollary~\ref{cor-simple}.
Thus Proposition~\ref{prop-P decomp} implies that the rank jump of the 
$A$-hypergeometric system $\HH_0(S_A,\beta') = M_A(\beta')$ is 
$j(\beta') = 4$.
\end{example}

When $d=3$ and 
$P^{\beta} = P_{\cJ(\widehat{b})}^{\beta}$ for some 
$\widehat{b}\in\widehat{B}^{\beta}$, 
\cite[Theorem~2.6]{okuyama} implies that the rank jump $j(\beta)$ of $M$ 
at $\beta$ corresponds to the reduced homology of the lattice $\FF(\beta)$.
The formula given by Okuyama involves this homology and the
volumes of the 1-dimensional faces of $A$ in $\FF(\beta)$.
For higher-dimensional cases, the cellular structure of the complex 
$I_{\cJ(\beta)}^\bullet$ of Notation~\ref{not-cpx I} shows that, in general, 
more information than the reduced homology of $\FF(\beta)$ is needed to 
compute $j(\beta)$, or even a single $j_{\widehat{b}}(\beta)$. 
 
Recall from Definition~\ref{def-ranking slab} that a ranking slab of $M$ 
is a stratum in the coarsest stratification of $\EE_A(M)$ 
that respects a specified collection of subspace arrangements. 
We are now prepared to prove Corollary~\ref{cor-strat vs equiv}, 
which states that the ranking slab stratification of $\EE_A(M)$ 
refines its rank stratification. 
From this it follows that each 
$\EE_A^i(M)=\{\beta\in\CC^d\mid j(\beta)>i\}$ 
is a union of ranking slabs, 
making each a union of translated linear subspaces of $\C^d$. 

\begin{proof}[Proof of Corollary~\ref{cor-strat vs equiv}]
If $\beta,\beta'\in\C^d$ belong to the same ranking slab, then 
the ranking lattices $\bbE^{\beta}=\bbE^{\beta'}$ coincide by 
Proposition~\ref{prop-zc E}. 
By Theorem~\ref{thm-compute jump}, the rank jumps
$j(\beta)$ and $j(\beta')$ coincide as well.
\end{proof}

\begin{corollary}\label{cor-EEA union linear translates}
For all integers $i\geq0$, $\EE_A^i(M)$ is a union of translates of linear subspaces 
that are generated by faces of $A$. 
\end{corollary}
\begin{proof}
This is an immediate consequence of Theorem~\ref{thm-str comparison} and
Corollary~\ref{cor-strat vs equiv}.
\end{proof}

The following is the second example promised 
at the end of Section~\ref{sec:exceptional arrangement}, 
showing that the rank of $\HH_0(M,\beta)$ need not be constant on a slab 
(see Definition~\ref{def:slab}). 
Further, this example shows that 
neither the arrangement stratification of $\EE_A(M)$ 
nor its refinement given by the $\ext$ modules in \eqref{13} 
determine its rank stratification.

\begin{example}\label{example-nonconstant slab}
Consider the matrix
\[
A = \left[\begin{array}{ccccccccccc}
    2 & 3 & 2 & 2 & 0 & 0 & 0 & 0 & 2 & 5 & 3 \\
    0 & 0 & 0 & 0 & 2 & 3 & 2 & 2 & 2 & 3 & 5 \\
    0 & 0 & 1 & 2 & 0 & 0 & 1 & 2 & 5 & 7 & 7
\end{array}\right]
\]
with $\vol(A) = 185$, and label the faces $F_1  =  [a_1\ a_2\ a_3\ a_4],$ 
$F_2  =  [a_5\ a_6\ a_7\ a_8]$, and $F_3  =  [a_9]$.
With $\beta' = \left[\begin{smallmatrix} 1 \\ 1 \\ 0 \end{smallmatrix}\right]$ 
and 
$\cP_A = \left\{
    \left[\begin{smallmatrix} 2 \\ 3 \\ 3 \end{smallmatrix}\right],
    \left[\begin{smallmatrix} 3 \\ 2 \\ 3 \end{smallmatrix}\right],
    \left[\begin{smallmatrix} 3 \\ 5 \\ 3 \end{smallmatrix}\right],
    \left[\begin{smallmatrix} 5 \\ 3 \\ 3 \end{smallmatrix}\right],
    \left[\begin{smallmatrix} 3 \\ 3 \\ 5 \end{smallmatrix}\right],
    \left[\begin{smallmatrix} 5 \\ 5 \\ 6 \end{smallmatrix}\right]
    \right\}$,
the exceptional arrangement of $S_A$ is
\[
\EE_A(S_A) = (\beta' + \CC F_3) \cup \cP_A.
\]
For $\beta \in \R_{\geq0} A \cap \EE_A(M)$,
\[
\cR_A(S_A,\beta) =
    \begin{cases}
    \beta & \text{if $\beta\in\cP_A$,}\\
    \beta'+\CC F_3 &
    \text{if $\beta \in [\beta' + \CC F_3]\setminus \beta'$,}\\
    \bigcup_{i=1}^3 [\beta' + \CC F_i]
    & \text{if $\beta = \beta'$,}\\
    \end{cases}
\]
so by the proof of Theorem~\ref{thm-main}, 
the rank jump of $M$ at $\beta\in\EE_A(M)$ is 
\[
j(\beta) =
    \begin{cases}
    1 & \text{if $\beta\in[\beta'+\CC F_3]\setminus\beta'$,}\\
    2 & \text{otherwise.}
    \end{cases}
\]
Here the arrangement stratification of $\EE_A(S_A)$ agrees with 
the one given by the $\ext$ modules that determine it, 
but $j(-)$ is not constant on the slab 
$[\beta'+\C F_3]\subseteq\EE_A(S_A)$. 
\end{example}\smallskip

To show that all of the arrangements in the definition of ranking slabs 
(Definition~\ref{def-ranking slab}) are necessary to obtain a refinement 
of the rank stratification of $\EE_A(M)$, we include 
the following example. Here, $j(\beta)$ changes 
where components of $\EE_A(S_A)$ that 
correspond to different $\ext$ modules intersect.

\begin{example}\label{plane-line-4diml}
Let 
\[
A = \left[\begin{array}{ccccccccccc}
    2 & 3 & 0 & 0 & 0 & 0  & 0 & 1 & 0 & 1 \\
    0 & 0 & 2 & 3 & 0 & 1  & 0 & 0 & 1 & 1 \\
    0 & 0 & 0 & 0 & 1 & 1  & 0 & 0 & 0 & 0 \\
    0 & 0 & 0 & 0 & 0 & 0  & 1 & 1 & 1 & 1 
\end{array}\right],
\]
$F  =  [a_1\ a_2]$, $G  =  [a_3\ a_4\ a_5\ a_6]$, and 
$\beta' = \left[\begin{smallmatrix} 1 \\ 1 \\ 0 \\ 0 \end{smallmatrix}\right]$. 
Here $\vol(A) = 21$, and 
the exceptional arrangement of $S_A$ is
\[
\EE_A(S_A) = [\beta' + \C F] \cup [\beta' + \C G], 
\]
where 
\[
- \qdeg \left( \ext_{R}^i(S_A,R)(-\varepsilon_A) \right) = 
\begin{cases}
    b+\C G & \text{ if $i = 7$,}\\
    b+\C F & \text{if $i = 8$,}\\
    \nothing & \text{if $i > 8$.} 
\end{cases}
\]
By Corollary~\ref{cor-simple} and Example~\ref{example-2 cmpts}, 
\[
j(\beta) = 
    \begin{cases}
    9 & \text{if $\beta = \beta'$,} \\
    6 & \text{if $\beta\in[\beta'+\C F]\setminus\beta'$,} \\
    4 & \text{if $\beta\in[\beta'+\C G]\setminus\beta'$.}
    \end{cases}
\]
\end{example}

We include a final example to show that $j(\beta)$ is 
not determined simply by $\wt{\N A}\setminus \N A$, 
the holes in the semigroup $\N A$. 

\begin{example}\label{ex-first example}
The matrix 
\[
A=\left[\begin{matrix}
    2&3&0&0&1&0&1\\
    0&0&2&3&0&1&1\\
    0&0&0&0&1&1&1
    \end{matrix}\right]
\]
has volume 16. 
The exceptional arrangement of $M=S_A$
is the union of four lines and a point: 
\[
\EE_A(S_A) =
    \left(
    \left[\begin{smallmatrix}
    1\\ 1\\ 0
    \end{smallmatrix}\right]
     + \C F
    \right)
    \cup
    \left(
    \left[\begin{smallmatrix}
    1\\ 1\\ 0
    \end{smallmatrix}\right]
     + \C G
    \right)
    \cup
    \left(
    \left[\begin{smallmatrix}
    1\\ 0\\ 0
    \end{smallmatrix}\right]
     + \C[a_5]
    \right)
    \cup
    \left(
    \left[\begin{smallmatrix}
    0\\ 1\\ 0
    \end{smallmatrix}\right]
     + \C[a_6]
    \right)
    \cup
    \left\{
    \left[\begin{smallmatrix}
    2\\ 2\\ 1
    \end{smallmatrix}\right]
    \right\}, 
\]
where $F = [a_1\ a_2]$ and $G = [a_3\ a_4]$. 
The generic rank jumps along each component are as follows: 
\begin{align*}
    \left[\begin{smallmatrix}
    1\\ 1\\ 0
    \end{smallmatrix}\right]
     + \C F
     \mapsto  3, \quad  
    \left[\begin{smallmatrix}
    1\\ 0\\ 0
    \end{smallmatrix}\right]
     + \C[a_5]
     \mapsto  1, \quad 
    \left[\begin{smallmatrix}
    1\\ 1\\ 0
    \end{smallmatrix}\right]
     + \C G
     \mapsto  3, \quad
    \left[\begin{smallmatrix}
    0\\ 1\\ 0
    \end{smallmatrix}\right]
     + \C[a_6]
     \mapsto  1, \quad
    \text{and }
    \left[\begin{smallmatrix}
    2\\ 2\\ 1
    \end{smallmatrix}\right]
     \mapsto  2.  
\end{align*}
These generic rank jumps are achieved everywhere except at the points
\[
    \left[\begin{smallmatrix}
    1\\ 1\\ 0
    \end{smallmatrix}\right],
    \left[\begin{smallmatrix}
    1\\ 0\\ 0
    \end{smallmatrix}\right],
    \left[\begin{smallmatrix}
    0\\ 1\\ 0
    \end{smallmatrix}\right], \text{ and }
    \left[\begin{smallmatrix}
    \phantom{-}0\\ \phantom{-}0\\ -1
    \end{smallmatrix}\right].
\]
The point that may be unexpected in this collection is 
	$b = \left[\begin{smallmatrix} 
	\phantom{-}0\\ \phantom{-}0\\ -1
	\end{smallmatrix}\right]$. 
Both of the components of $\EE_A(S_A)$ that contain $b$ 
have generic rank jumps of 1; 
however, $j(b) = 2$. 
This is because the ranking arrangement of $S_A$ 
has three components that contain $b$: 
\begin{eqnarray}\label{eqn-R b 1st ex}
\cR_A(S_A, b) =
    \left(    b + \C[a_5]    \right)	    \cup
    \left(    b + \C[a_6]    \right)    \cup
    \left(    b + \C[a_1\ a_2\ a_3\ a_4]    \right). 
\end{eqnarray}
If the plane $(b + \C[a_1\ a_2\ a_3\ a_4])$ were not in $\cR_A(S_A,\beta)$, 
then the rank jump of $S_A$ at $b$ would only be 1 by \eqref{24}. 
Thus the hyperplane in \eqref{eqn-R b 1st ex}, although unrelated to the holes in 
$\N A$, accounts for the higher value of $j(b)$. 

The rank jumps at the other parameters are \
\[
j\left( S_A, 
    \left[\begin{smallmatrix}
    1\\ 0\\ 0
    \end{smallmatrix}\right] \right) =
j\left( S_A, 
    \left[\begin{smallmatrix}
    0\\ 1\\ 0
    \end{smallmatrix}\right]
\right) = 3
\quad \text{ and } \quad
j\left( S_A, 
    \left[\begin{smallmatrix}
    1\\ 1\\ 0
    \end{smallmatrix}\right]
\right) = 5.
\]

\end{example}

The algebraic upper semi-continuity of the rank of $\HH_0(M,\beta)$ implies that 
most of the codimension 1 components of the ranking arrangement 
$\cR_A(M)$ do not increase the rank of $\HH_0(M,\beta)$. 
It would interesting to know if the set of such hyperplanes can be identified. 

\section{The isomorphism classes of $A$-hypergeometric systems}
\label{sec:isom classes}

When $M=S_A$, the results of Section~\ref{sec:comb of rank} apply to 
the $A$-hypergeometric system $\HH_0(S_A,\beta) = M_A(\beta)$.
For a face $\tau$ of $A$, 
\[
E_\tau(\beta)
=\{\lambda\in\C\tau\mid\beta-\lambda\in\N A+\Z\tau\}/\Z\tau
\]
is a finite set.
It is shown in \cite{saito isom} and \cite{ST diffl algs} that $M_A(\beta)$ and 
$M_A(\beta')$ are isomorphic as $D$-modules precisely when
$E_\tau(\beta)=E_\tau(\beta')$ for all faces $\tau$ of $A$. 
We will now use Euler--Koszul homology to give a simple proof of one 
direction of this equivalence; first we exhibit a complementary
relationship between $E_\tau(\beta)$ and  $\bbE_\tau^{\beta}$ 
(see \eqref{eqn-partition bbE}). 

\begin{observation}\label{obs-Etau}
It is shown in Theorem~\ref{thm-compute jump} 
that as $\tau$ runs through the faces of $A$, the sets
\begin{eqnarray*}
\bbE_\tau^{\beta} = \Z^d \cap 
(\beta+\C\tau)\setminus(\N A+\Z\tau)
= \beta - \{\lambda \in\C\tau
    \mid \beta-\lambda\in\Z^d\setminus(\N A + \Z\tau) \},
\end{eqnarray*}
determine the rank jump of $M_A(\beta)$ at $\beta$. 
Notice that $(\beta - \bbE_\tau^{\beta})/\Z\tau$ is the
complement of $E_\tau(\beta)$ in the group
$(\Z^d\cap\Q\tau)/\Z\tau$.
\end{observation}

\begin{lemma}\label{lemma-Nisom}
If $\beta,\beta'\in\C^d$ are such that $\beta'-\beta = A\lambda$ 
for some $\lambda\in\N^n$, then the map defined by right multiplication
$\del_A^{\lambda}:M_A(\beta)\rightarrow M_A(\beta')$
is an isomorphism of $D$-modules if and only if
$\beta' \notin \qdeg\left( S_A/\<\del_A^{\lambda}\> \right).$
\end{lemma}
\begin{proof}
See \cite[Remark~3.6]{EKDI}.
\end{proof}

For a vector $v\in\C^n$, let the {\it support} of $v$ be the subset 
$\supp(v) = \{a_i\mid v_i\neq0\}$ of the columns of $A$.

\begin{lemma}\label{lemma-NEtau}
If $\beta,\beta'\in\C^d$ are such that $\beta'-\beta = A\lambda$ 
for some $\lambda\in\N^n$ and $E_\tau(\beta)=E_\tau(\beta')$ 
for all faces $\tau$ of $A$, then the map defined by right multiplication
$\del_A^{\lambda}:M_A(\beta)\rightarrow M_A(\beta')$ is an isomorphism.
\end{lemma}
\begin{proof}
If $\del_A^\lambda$ is not an isomorphism, then
$\beta'\in\qdeg(S_A/\<\del_A^\lambda\>)$ by Lemma~\ref{lemma-Nisom}.
By the definition of quasidegrees, there exist vectors $v\in\N^n$ and 
$\gamma\in\C\tau$ for some face $\tau$ such that $\beta' = Av+\gamma$, 
$\supp(\lambda)\not\subseteq\tau$, and $v_i \leq \lambda_i$ for all $i$. 
Hence $\beta' - \gamma\in\N A$, 
so $\gamma+\Z\tau\in E_\tau(\beta')$.
Further, the condition $v_i\leq \lambda_i$ for all $i$ implies that
$\gamma+\Z\tau\notin E_\tau(\beta)$.
\end{proof}

\begin{notation}
If a vector $\lambda\in\N^n$ is such that the map given by right multiplication 
$\del_A^\lambda:M_A(\beta)\rightarrow M_A(\beta+A\lambda)$ is an 
isomorphism, let $\del_A^{-\lambda}$ denote its inverse.
\end{notation}

\begin{theorem}\label{thm-isom classes}
The $A$-hypergeometric systems $M_A(\beta)$ and $M_A(\beta')$ are 
isomorphic if and only if $E_\tau(\beta) = E_\tau(\beta')$ for all faces $\tau$ of $A$.
\end{theorem}
\begin{proof}
The proof of the ``only-if" direction holds as in \cite{saito isom} without a 
homogeneity assumption on $A$ because it involves 
only the construction of formal solutions of $M_A(\beta)$.

For the ``if" direction, suppose that
$E_\tau(\beta) = E_\tau(\beta')$ for all faces $\tau$ of $A$.
As stated in \cite[Proposition~2.2]{saito isom},
$E_A(\beta)=E_A(\beta')$ implies by definition that $\beta'-\beta\in\Z A$,
so $\beta'-\beta = A\lambda$ for some $\lambda\in\Z^n$.
There are unique vectors $\lambda_+, \lambda_-\in\N^n$ with disjoint support 
such that $\lambda = \lambda_+ - \lambda_-$. 
In light of Lemma~\ref{lemma-NEtau}, we may assume that both 
$\lambda_+$ and $\lambda_-$ are nonzero. 
We claim that at least one of $\del_A^{-\lambda_-}\del_A^{\lambda_+}$ 
or $\del_A^{\lambda_+}\del_A^{-\lambda_-}$ defines an isomorphism 
from $M_A(\beta)$ to $M_A(\beta')$.

If $\del_A^{\lambda_+}:M_A(\beta)\rightarrow M_A(\beta+A\lambda_+)$ or
$\del_A^{\lambda_+}:M_A(\beta-A\lambda_-)\rightarrow M_A(\beta')$ 
defines an isomorphism, then the ``only-if" direction and 
Lemma~\ref{lemma-NEtau} imply that
$\del_A^{-\lambda_-}\del_A^{\lambda_+}$ or
$\del_A^{\lambda_+}\del_A^{-\lambda_-}$, respectively,
give the desired isomorphism.
We are left to consider the case when
$\del_A^{\lambda_+}$ does not define an isomorphism from either domain.
By Lemma~\ref{lemma-Nisom}, this is equivalent to
\begin{eqnarray}\label{eq:two things}
\beta+A\lambda_+ \in\qdeg(S_A/\<\del_A^{\lambda_+}\>) 
\quad \text{\ and\ }\quad 
\beta'  \in\qdeg(S_A/\<\del_A^{\lambda_+}\>).
\end{eqnarray}
From the right side of \eqref{eq:two things}, we see that the nonempty face 
$\eta :=\supp(\lambda_-)$ is such that 
$\beta' +\C \eta\subseteq \qdeg(S_A/\<\del_A^{\lambda_+}\>)$, 
so $E_\eta(\beta')\neq \nothing$. 
However, the shift $A\lambda_+$ in the left side of \eqref{eq:two things} 
implies that $(\beta+\C\eta)\cap (\N A + \Z \eta) = \nothing$. 
Thus $E_\tau(\beta) = \nothing$, which is a contradiction. 
\end{proof}

It is not yet understood how the holomorphic solution
space of $M_A(\beta)$ varies as a function of $\beta$;
different functions of $\beta$ suggest alternative behaviors.
Walther showed in \cite{uli monodromy} that the reducibility of the monodromy
of $M_A(\beta)$ varies with $\beta$ in a lattice-like fashion. 
When the convex hull of $A$ and the origin is a simplex,
Saito used the sets $E_\tau(\beta)$ to construct 
a basis of holomorphic solutions of $M_A(\beta)$ with 
a common domain of convergence \cite{saito log-free}.
Thus Theorem~\ref{thm-compute jump} 
and the complementary relationship in Observation~\ref{obs-Etau} 
between the $E_\tau(\beta)$ and the ranking lattices of $S_A$ at $\beta$ 
suggest that the ranking slabs give the coarsest stratification 
over which there could be a 
constructible sheaf of solutions for the hypergeometric system. 

\raggedbottom
\def\cprime{$'$} \def\cprime{$'$}
\providecommand{\MR}{\relax\ifhmode\unskip\space\fi MR }
\providecommand{\MRhref}[2]{\href{http://www.ams.org/mathscinet-getitem?mr=#1}{#2}}
\providecommand{\href}[2]{#2}

\end{document}